\documentclass[12pt]{article}
\usepackage{amsmath,amssymb}
\usepackage{amsthm}
\usepackage{graphicx}
\usepackage{ascmac}
\usepackage{subcaption}
\usepackage{mathtools}
\usepackage{empheq}
\usepackage[margin=35mm]{geometry}

\mathtoolsset{showonlyrefs = true}
\newtheorem{theorem}{Theorem}[section]
\newtheorem{lemma}[theorem]{Lemma}
\newtheorem{corollary}[theorem]{Corollary}

\theoremstyle{definition}
\newtheorem{definition}{Definition}[section]
\newtheorem{remark}{Remark}[section]

\DeclareMathOperator{\w}{wd}
\DeclareMathOperator{\rt}{rot}
\DeclareMathOperator{\sign}{sign}

\title{Generalized Schr\"oder paths arising from a combinatorial 
interpretation of generalized Laurent bi-orthogonal polynomials}
\author{Mawo Ito\thanks{Department of Applied Mathematics and Physics, Graduate School of Infomatics, 
Kyoto University, Kyoto 606-8501, Japan.}}
\date{}

\begin{document}
\maketitle
\begin{abstract}
Lattice paths called $\ell$-Schr\"oder paths
are introduced.
They are paths
on the upper half-plane 
consisting of $\ell+2$ types of steps: $(i,\ell-i)$ for $i=0,\ldots,\ell$, and $(1,-1)$. 
Those paths generalize Schr\"oder paths and some variants, such as
$m$-Schr\"oder paths by Yang and Jiang and Motzkin-Schr\"oder paths by Kim and Stanton.
We show that $\ell$-Schr\"oder paths arise naturally from a combinatorial interpretation 
of the moments of generalized Laurent bi-orthogonal polynomials
introduced by Wang, Chang, and Yue.
We also show that some generating functions of
non-intersecting $\ell$-Schr\"oder paths can be factorized in closed forms.
\end{abstract}
\section{Introduction}
A Schr\"oder path is a path in the $x$-$y$ plane with three types of elementary steps $(1,0),(0,1)$, and $(1,-1)$,
never going below the $x$-axis.
Schr\"oder paths are among the well-studied lattice paths in combinatorics, 
having attracted attention due to their connection with exactly solvable models~\cite{FG,S,BDS}. 
A typical example of such a model is the domino tilings of the Aztec diamonds
introduced by Elkies et al.~\cite{EKLP}, 
which exhibits a one-to-one correspondence with non-intersecting tuples of 
Schröder paths~\cite{J}. 
Schr\"oder paths have
a notable property that
certain generating functions of 
non-intersecting tuplets are expressed in a closed form~\cite{K2,EuFu}, which
 is one of the critical
ingredients to enumerate solvable models.

In this paper, 
we propose a generalization of 
Schr\"oder paths that preserves this notable property of generating functions.
\begin{definition}
	\label{definition:1}
	For a positive integer $\ell$, an \emph{$\ell$-Schr\"oder path} is a lattice path in the 
	$x$-$y$ plane with steps 
	$\mathsf{a}_i\triangleq(i,\ell-i)$ for $i=0\ldots,\ell$ and $\mathsf{a}_{\ell+1}\triangleq(1,-1)$,
	never going below the $x$-axis.
\end{definition}
\begin{figure}[htbp]
  \begin{minipage}[b]{0.79\linewidth}
    \centering
	\includegraphics{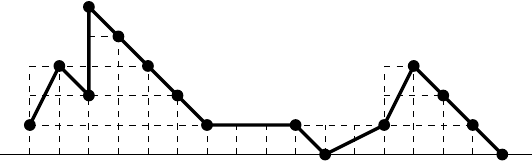}
  \end{minipage}
  \begin{minipage}[b]{0.20\linewidth}
    \centering
	\includegraphics{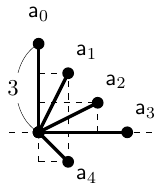}
  \end{minipage}
    \caption{
		Left: an example of $3$-Schr\"oder path.
	Right: the five types of steps in $3$-Schr\"oder path. }
  \label{fig:Schroder}
\end{figure}
See Figure \ref{fig:Schroder} for an example of $\ell$-Schr\"oder paths when $\ell=3$.

In the literature,
there are several ways to extend Schr\"oder paths.
One way is to use steps $(0,m)$ and $(1,m-1)$ instead of $(0,1)$ and $(1,0)$.
Such a generalization is called $m$-Schr\"oder paths in \cite{YJ}.
Another way is to add a new step $(1,1)$ to Schr\"oder paths.
Such paths are called Motzkin-Schr\"oder paths and were introduced by Kim and Stanton \cite{KimStanton}
in connection with orthogonal functions.
In \cite{KimStanton}, they found that certain generating functions of non-intersecting tuplet of 
Motzkin-Schr\"oder paths admit a closed-form factorization.
The $\ell$-Schr\"oder paths in this paper essentially contain both $m$-Schr\"oder and 
Motzkin-Schr\"oder paths as special cases.

The aim of this paper is twofold: 
(i) to describe that $\ell$-Schr\"oder paths arise naturally from a
combinatorial interpretation of orthogonal functions;
(ii) to show that
certain generating functions of non-intersecting tuplets of $\ell$-Schr\"oder paths
are written in closed forms.
See Figure \ref{fig:PiOmega} in Section \ref{section:4} for an example of 
a non-intersecting tuplet that
we will consider.

The combinatorial theory of orthogonal polynomials (OPs) was established by Viennot~\cite{V}
using weighted paths.
A similar approach was developed by Kamioka~\cite{K1,K2}
for Laurent bi-orthogonal polynomials (LBPs),
the analogue of OPs introduced in the field of 
two-point Pad\'e approximation and discrete integrable systems~\cite{Z}.
In a series of his studies, Kamioka interpreted LBPs using weighted Schr\"oder paths.


LBPs are, as in OPs, characterized by so-called the fundamental three-term recurrence relation \cite{Z}
\begin{align}
	P_{n+1}(x)=(x-d_n)P_n(x)-b_nxP_{n-1}(x),
\end{align}
where $d_n$ and $b_n$ are non-zero.
Generalizations of LBPs are obtained by increasing the number of terms in the recurrence as
\begin{align}
	\label{eq:recurrence}
	P_n(x)=-\sum_{i=1}^{\ell-1}a_{i,n-\ell}P_{n-i}(x)
	+(x^{\ell}-a_{\ell,n-\ell})P_{n-\ell}(x)
	-a_{\ell+1,n-\ell}x^{\ell}P_{n-\ell-1}(x),\quad n\geq \ell,
\end{align}
where $a_{\ell,i}$ and $a_{\ell+1,i}$ are non-vanishing, and the other $a_{i,j}$
are some constants.
Such generalized LBPs are studied by several authors~\cite{WCY,KK,BV}.
We refer to these generalized LBPs as $\ell$-LBPs
since the letter $\ell$ is used as a fixed parameter in \cite{WCY}.
We show that $\ell$-Schr\"oeder paths emerge
from a combinatorial interpretation of $\ell$-LBPs.

To show the second aim of this paper, 
we employ the technique developed by Eu and Fu \cite{EuFu}. 
They enumerate non-intersecting tuplets of Schr\"oder paths 
in closed form
by combining two kinds of paths, i.e., Schr\"oder and small Schr\"oder paths.
We define new paths called small $\ell$-Schr\"oder paths and generalize their technique.

This paper is organized as follows:
In Section \ref{section:2}, we recall the definition and fundamentals of 
$\ell$-LBPs and present a Favard-type theorem 
essential for a combinatorial interpretation.
In Section \ref{section:3}, 
we interpret the moments of $\ell$-LBPs using $\ell$-Schr\"oder paths (Theorem \ref{theorem:main1}).
This theorem parallels the Kamioka development~\cite{K1} for conventional LBPs.
In section \ref{section:4},
we examine the enumerative aspects of $\ell$-Schr\"oder paths. 
We explore the relationship between $\ell$-Schr\"oder paths and 
small $\ell$-Schr\"oder paths. 
Using the methods of Eu and Fu~\cite{EuFu}, 
we express generating functions of non-intersecting $\ell$-Schr\"oder paths in a closed form
(Theorem \ref{theorem:Pi_explicit}).
As an application of the results in this chapter, we discuss 
specialized generating functions of $\ell$-Schr\"oder paths, namely the 
$\ell$-Narayana polynomials.

\section{$\ell$-LBPs and Favard-type theorem}
\label{section:2}
In this section, we briefly review the fundamentals of $\ell$-LBPs, including 
their definitions, recurrence relations, Favard-type theorem, and determinant expressions.

Let $\mathbb{K}$ be a field, and $\mathbb{N}_0$ be 
the set of non-negative integers.
Let $\mathcal{L}\colon\mathbb{K}[x,x^{-1}]\to\mathbb{K}$
be a linear functional on the space of Laurent polynomials over $\mathbb{K}$.
A sequence of monic polynomials $(P_n(x))_{n=0}^\infty$
with each $P_n(x)$ having exact degree $n$
is called an 
\emph{$\ell$-Laurent 
bi-orthogonal polynomials ($\ell$-LBPs)} with respect to $\mathcal{L}$ if
it satisfies the orthogonality
\begin{align}
	\label{eq:ellLBP_orthogonality}
\mathcal{L}\left[P_n(x)x^{-k\ell}\right]=h_n\delta_{nk},\qquad n,k\in\mathbb{N}_0,\ 0\leq k\leq n
\end{align}
for some non-vanishing constant $h_n\neq 0$.
The symbol $\delta_{n,k}$ denotes the Kronecker delta.
In the case $\ell=1$, the conventional LBPs are recovered.

Each polynomial $P_n(x)$ in $\ell$-LBPs satisfies the 
$(\ell+2)$-term recurrence relation \eqref{eq:recurrence} for $n\geq\ell$ (cf. \cite{WCY}).
Conversely, the following theorem states that any sequence of 
polynomials defined by the recurrence \eqref{eq:recurrence} 
with suitable initial condition 
is a $\ell$-LBPs.
\begin{theorem}[Favard-type theorem for $\ell$-LBPs]
	\label{theorem:Favard}
	Let $a_{i,j}\in\mathbb{K}$ be constants such that 
	$a_{\ell,i}\neq 0$ and $a_{\ell+1,i}\neq 0$ for all $i\in\mathbb{N}_0$.
	Let $p_n(x)\in\mathbb{K}[x]$ be monic polynomials of degree $n$ for $n=0,\ldots,\ell-1$.
	The sequence of polynomials $(P_n(x))_{n=0}^\infty$ defined by 
	the recurrence \eqref{eq:recurrence} with initial conditions $P_{-1}(x)\triangleq 0$ and $P_n(x)\triangleq p_n(x)$ for 
$n=0,\ldots,\ell-1$, 
	is an $\ell$-LBPs with respect to some linear functional $\mathcal{L}$.
	Furthermore, if we impose the condition 
	$\mathcal{L}[1]=1$, such a linear functional $\mathcal{L}$ is uniquely determined.
\end{theorem}
Theorem \ref{theorem:Favard} is fundamental in the theory of $\ell$-LBPs. However,
the author cannot find the reference that explicitly provides its proof, so 
we prove it here.
We need a lemma to prove Theorem \ref{theorem:Favard}.
\begin{lemma}
	\label{lemma:det_lemma1}
	Let $(P_n(x))_{n=0}^\infty$ be as in Theorem \ref{theorem:Favard}.
	For any non-negative integer $k\in\mathbb{N}_0$, we have
 	\begin{align}
	\label{eq:det_lemma1}
		\det\left([x^j]P_{k+i}(x)\right)_{i,j=0}^{\ell-1}=(-1)^{k\ell}\prod_{i=0}^{k-1}a_{\ell,i},
	\end{align}
	where $[x^j]P_{k+i}(x)$ denotes the coefficients of $x^{j}$ in the polynomial $P_{k+i}(x)$.
\end{lemma}
\begin{proof}
	Let 
	$T_k\triangleq \det([x^j]P_{k+i}(x))_{i,j=0}^{\ell-1}$
	be the left-hand side of \eqref{eq:det_lemma1}.
	We derive a recurrence for $T_k$.
	It is easy to check that $T_0=1$ by using the monicity of $p_k(x)$ for $k=0,\ldots,\ell-1$.

	Let us think about $T_k$ for $k\geq 1$.
	By multiplying the first column by $a_{\ell,k}$, we get
	\begin{align}
		a_{\ell,k}T_k=\det
		\left[
			\begin{array}{cccc}
				\left[x^0\right]a_{\ell,k}P_k(x)& \left[x^0\right]P_{k+1}(x)&
				\cdots &\left[x^0\right]P_{k+\ell-1}(x)\\
				\vdots&&&\vdots\\
				\left[x^{\ell-1}\right]a_{\ell,k}P_k(x)& \left[x^{\ell-1}\right]P_{k+1}(x)&
				\cdots &\left[x^{\ell-1}\right]P_{k+\ell-1}(x)
			\end{array}
		\right].
	\end{align}
	Multiplying the $j$-th column by $a_{\ell-j,k}$ and adding it to the first ($0$-th)
	column for $j=1,\ldots,\ell-1$
	we have
	\begin{align}
		\label{eq:lemma2_proof}
		a_{\ell,k}T_k=\det
		\left[
			\begin{array}{cccc}
				\left[x^0\right]F_k& \left[x^0\right]P_{k+1}(x)&
				\cdots &\left[x^0\right]P_{k+\ell-1}(x)\\
				\vdots&&&\vdots\\
				\left[x^{\ell-1}\right]F_k& \left[x^{\ell-1}\right]P_{k+1}(x)&
				\cdots &\left[x^{\ell-1}\right]P_{k+\ell-1}(x)
			\end{array}
		\right],
	\end{align}
	where we write $F_k\triangleq \sum_{i=1}^\ell a_{i,k}P_{k+\ell-i}$.
	From the recurrence \eqref{eq:recurrence}, we have
	$[x^i]F_k=-[x^i]P_{k+\ell}$
	for $k\geq 0$ and $i=0,\ldots,\ell-1$.
	Substituting this into 
		\eqref{eq:lemma2_proof}, 
	we get the recurrence $T_{k+1}=(-1)^\ell a_{\ell,k}T_k$ for $k\geq 0$.
\end{proof}

\begin{proof}[Proof of Theorem \ref{theorem:Favard}]
	We define a linear functional $\mathcal{L}$ by defining its moments 
	$\mu_n=\mathcal{L}[x^n]\ (n\in\mathbb{Z})$ inductively on $n$.
	Let us define the {\it positive-side} moments $\mu_{n}\ (n\geq 0)$ by
	$\mathcal{L}[1]=1$ and equations
	\begin{align}
		\label{eq:Favard_positive}
		\mathcal{L}[P_n(x)]=0,\qquad n=1,2,3,\ldots.
	\end{align}
	Also, define the {\it negative-side} moments $\mu_{n}\ (n<0)$ by solving
	the system of equations 
	\begin{align}
		\label{eq:Favard_negative}
		\mathcal{L}[P_{n+i}(x)x^{-n\ell}]=0,\qquad i=1,\ldots,\ell
	\end{align}
	for $n=1,2,\ldots$,
	as Lemma \ref{lemma:det_lemma1} shows 
	the system \eqref{eq:Favard_negative} has a unique solution.
	Note that the moments $\mu_{-\ell n+i}\ (i=0,1,\ldots,\ell-1)$ are determined from equations \eqref{eq:Favard_negative}.
	This process determines the moments $\mu_{n}\ (n\in\mathbb{Z})$ uniquely.

	Next, 
	we show that the linear functional $\mathcal{L}$ defined by the moments $\mu_n\ (n\in\mathbb{Z})$
	as above 
	satisfies the orthogonality
	\begin{align}
		\label{eq:Favard_orth1}
		\mathcal{L}[P_n(x)x^{-k\ell}]&=0,\qquad k,n\in\mathbb{N}_0, k< n,\ {\rm and}\\
		\label{eq:Favard_orth2}
		\mathcal{L}[P_n(x)x^{-n\ell}]&\neq 0,\qquad n\in\mathbb{N}_0.
\end{align}
The equations \eqref{eq:Favard_orth1} hold for $k=0$ and $k=n-\ell,\ldots,n-1$ by the definition of 
the moments.
In other words, the equation \eqref{eq:Favard_orth1} is already shown when  $n\leq \ell+1$.
Therefore, we show the remaining cases $n\geq \ell+2$ by induction on $n$. 
Assume that \eqref{eq:Favard_orth1} 
holds for all $n$ less than $N$, where $N\geq\ell+2$.
Using the recurrence \eqref{eq:recurrence}, one can write $\mathcal{L}[P_N(x)x^{-k\ell}]$ as 
a linear combination of 
$\mathcal{L}[P_{N-i}(x)x^{-k\ell}]\ (i=1,\ldots,\ell-1),\ \mathcal{L}[P_{N-\ell}(x)x^{-(k-1)\ell}],\ 
\mathcal{L}[P_{N-\ell}(x)x^{-k\ell}]$, and $\mathcal{L}[P_{N-\ell-1}(x)x^{-(k-1)\ell}]$, 
all of which are zero by the inductive hypothesis.
Thus \eqref{eq:Favard_orth1} holds when $n=N$.

The equation \eqref{eq:Favard_orth2} is shown by
a recurrence
\begin{align}
	\label{eq:Favard_orth3}
	\mathcal{L}[P_n(x)x^{-n\ell}]=-\frac{a_{\ell+1,n}}{a_{\ell,n}}\mathcal{L}[P_{n-1}(x)x^{-(n-1)\ell}],
	\qquad n\geq 1,
\end{align}
which is derived from multiplying both sides of \eqref{eq:recurrence} by $x^{-(n-\ell)\ell}$ and acting $\mathcal{L}$.
\end{proof}

Additionally, we would like to discuss the relationship between the moments of $\ell$-LBPs
defined by the recurrence \eqref{eq:recurrence} 
with different initial conditions.
We define two $\ell$-LBPs as equivalent if they satisfy 
the recurrence \eqref{eq:recurrence} with identical coefficients $a_{i,j}$. 
The set of all $\ell$-LBPs is divided into equivalence classes based on this equivalence relation. 
Within each equivalence class, there exists a unique $\ell$-LBPs $(P_n(x))_{n=0}^\infty$ 
whose first $\ell$ polynomials are monomials, i.e., 
$P_i(x)=x^i$ for $i=0,\ldots,\ell-1$. 
We refer to such $\ell$-LBPs as \emph{primitive}. 
The following theorem states that the moments of arbitrary $\ell$-LBPs can be 
expressed as a linear combination of the moments of equivalent primitive $\ell$-LBPs.

\begin{theorem}
	\label{theorem:moment_primitive}
	Let $(P_n(x))_{n=0}^\infty$ be a primitive $\ell$-LBPs, and let 
	$(\widetilde{P}_n(x))_{n=0}^\infty$ be a equivalent $\ell$-LBPs
	with first $\ell$ polynomials defined as 
	\begin{align}
		(\widetilde{P}_0(x)\ \widetilde{P}_1(x)\ \widetilde{P}_2(x)\ \cdots\ \widetilde{P}_{\ell-1}(x))^t
		\triangleq
		S(1\ x\ x^2\ \cdots\ x^{\ell-1})^t,
	\end{align}
	where $S$ is an 
	arbitrary
	$\ell\times\ell$ lower triangular matrix with $1$'s on its diagonal.
	Furthermore,
	let $(\mu_n)_{n\in\mathbb{Z}}$ and $(\widetilde{\mu}_n)_{n\in\mathbb{Z}}$ represent
	the moments of 
	linear functional 
	with $\mu_0=\widetilde{\mu}_0=1$ 
	associated with
	$(P_n(x))_{n=0}^\infty$ and $(\widetilde{P}_n(x))_{n=0}^\infty$, respectively.
	Then, we have 
	\begin{align}
		\label{eq:def_new_moment}
		\left(
		\begin{array}{c}
			\widetilde{\mu}_{m\ell}\\
			\widetilde{\mu}_{m\ell+1}\\
			\vdots\\
			\widetilde{\mu}_{m\ell+\ell-1}
		\end{array}
		\right)
		=S^{-1}
		\left(
		\begin{array}{c}
			\mu_{m\ell}\\
			\mu_{m\ell+1}\\
			\vdots\\
			\mu_{m\ell+\ell-1}
		\end{array}
		\right)
	\end{align}
	for all $m\in\mathbb{Z}$.
\end{theorem}

We prove Theorem \ref{theorem:moment_primitive} after presenting a Lemma.
Let $\Delta\triangleq (\delta_{i+1,j})_{i,j=0}^\infty$ be the infinite sized upper shift matrix.
\begin{lemma}
	\label{lemma:Niimiya}
	Let $L$ be an infinite sized lower triangular matrix.
	The equation
	$L\Delta^{\ell}=\Delta^{\ell}L$ holds if and only if 
	$L$ is a block diagonal matrix of the form $L={\rm diag}(S,S,S,\ldots)$, where
	$S$ is an
	 $\ell\times\ell$ lower triangular matrix.
\end{lemma}
\begin{proof}
	We demonstrate the `only if' side; the `if' side is obvious.
	Let us express the matrix $L$ 
	as a block matrix 
	\begin{align}
		L=(L_{i,j})_{i,j=0}^\infty=
		\left(
			\begin{array}{ccccc}
				L_{0,0}&O&O&O&\cdots\\
				L_{1,0}&L_{1,1}&O&O&\cdots\\
				L_{2,0}&L_{2,1}&L_{2,2}&O&\cdots\\
				\vdots&\vdots&\vdots&\vdots&
		\end{array}
		\right),
	\end{align}
	where each $L_{i,j}$ is an $\ell\times\ell$ matrix.
	The equation $L\Delta^{\ell}=\Delta^{\ell}L$ can be expressed as 
	\begin{align}
		\left(
			\begin{array}{ccccc}
				O&L_{0,0}&O&O&\cdots\\
				O&L_{1,0}&L_{1,1}&O&\cdots\\
				O&L_{2,0}&L_{2,1}&L_{2,2}&\cdots\\
				\vdots&\vdots&\vdots&\vdots&
		\end{array}
		\right)
		=
		\left(
			\begin{array}{ccccc}
				L_{1,0}&L_{1,1}&O&O&\cdots\\
				L_{2,0}&L_{2,1}&L_{2,2}&O&\cdots\\
				L_{3,0}&L_{3,1}&L_{3,2}&O&\cdots\\
				\vdots&\vdots&\vdots&\vdots&
		\end{array}
		\right).
	\end{align}
	Comparing both sides, we obtain $L_{i,j}=L_{0,0}\delta_{i,j}$.
	This implies that $L$ takes the form $L={\rm diag}(S,S,S,\ldots)$ where $S=L_{0,0}$.
\end{proof}
\begin{proof}[Proof of Theorem \ref{theorem:moment_primitive}]
	Let ${\rm P}$ and $\widetilde{{\rm P}}$ be the coefficient matrices of the polynomials
	$(P_n(x))_{n=0}^\infty$ and $(\widetilde{P}_n(x))_{n=0}^\infty$, i.e., 
	\begin{align}
	{\rm P}\triangleq([x^j]P_i(x))_{i,j=0}^\infty\quad
	\textrm{and}\quad \widetilde{{\rm P}}\triangleq([x^j]\widetilde{P}_j(x))_{i,j=0}^\infty,
\end{align}
	where $[x^i]P_j(x)$ represents the coefficient of $x^i$ in $P_j(x)$.
	Let ${\rm S}\triangleq{\rm diag}(S,S,\ldots)$ be the infinite sized block diagonal matrix with 
	the matrix $S$'s on its diagonal.
	We first show that $\widetilde{{\rm P}}={\rm PS}$.
	
	The recurrence \eqref{eq:recurrence} can be written as a matrix equation 
	\begin{align}
		\label{eq:mateq1}
	{\rm BP}\Delta^{\ell}={\rm AP},
	\end{align}
	where
	${\rm B}$ is an infinite sized lower bidiagonal and ${\rm A}$ is an upper triangular band
	matrix defined as
	\begin{align}
		{\rm B}&\triangleq
		\left(
			\begin{array}{ccccc}
				1&&&&\\
				-a_{\ell+1,1}&1&&&\\
							 &-a_{\ell+1,2}&1&&\\
							 &&-a_{\ell+1,3}&1&\\
							 &&&\ddots&\ddots
			\end{array}
			\right),\\
			{\rm A}&\triangleq
		\left(
			\begin{array}{ccccccccc}
				a_{\ell,0}&a_{\ell-1,0}&a_{\ell-2,0}&\cdots&a_{1,0}&1&&\\
						  &a_{\ell,1}&a_{\ell-1,1}&\cdots&a_{2,1}&a_{1,1}&1&\\
						  &&a_{\ell,2}&\cdots&a_{3,2}&a_{2,2}&a_{1,2}&1\\
						  &&&\ddots&&&&&\ddots
			\end{array}
			\right).
	\end{align}
	The matrix $\widetilde{{\rm P}}$ also satisfies \eqref{eq:mateq1} since it
	satisfies the same recurrence \eqref{eq:recurrence}.
	This implies that ${\rm P}$ and $\widetilde{{\rm P}}$ satisfies 
	${\rm P}\Delta^{\ell}{\rm P}^{-1}=
	\widetilde{{\rm P}}\Delta^{\ell}\widetilde{{\rm P}}^{-1}=
	{\rm B}^{-1}{\rm A}$.
	By transforming the equation, we obtain 
	$\Delta^{\ell}{\rm P}^{-1}\widetilde{{\rm P}}=
	{\rm P}^{-1}\widetilde{{\rm P}}\Delta^{\ell}$.
	As ${\rm P}^{-1}\widetilde{{\rm P}}$ is lower triangular, from Lemma \ref{lemma:Niimiya},
	we have ${\rm P}^{-1}\widetilde{{\rm P}}={\rm diag}(M,M,\ldots)$ for some 
	$\ell\times\ell$ matrix $M$.
	Compareing the first $\ell\times\ell$ entries of the both sides of 
	$\widetilde{{\rm P}}={\rm P}{\rm diag}(M,M,\ldots)$,
	we conclude that $M=S$.

	Next, we demonstrate the statement of the Theorem.
	The orthogonality \eqref{eq:ellLBP_orthogonality} can be expressed by using matrices as
	\begin{align}
		\label{eq:PLH}
	{\rm PL}={\rm H},
\end{align}
where ${\rm L}$ is an infinite sized matrix 
	with moments in its entries 
	and ${\rm H}$ is an 
	infinite sized 
	upper triangular matrix 
	with $h_i\ (i=0,1,2,\ldots)$ on its diagonal,
	as
	\begin{align}
		{\rm L}&\triangleq(\mu_{i-\ell j})_{i,j=0}^{\infty}\\
			&=\left(
				   \begin{array}{cccc}
				\mu_0&\mu_{-\ell}&\mu_{-2\ell}&\cdots\\
				\mu_1&\mu_{-\ell+1}&\mu_{-2\ell+1}&\cdots\\
				\mu_2&\mu_{-\ell+2}&\mu_{-2\ell+2}&\cdots\\
				\vdots&\vdots&\vdots&
			\end{array}
			   \right),\\
		{\rm H}&\triangleq
			   \left(
				   \begin{array}{cccc}
				h_0\ast &\ast&\ast&\cdots\\
				0&h_1&\ast&\cdots\\
				0&0&h_2&\cdots\\
				\vdots&\vdots&\vdots&
			\end{array}
			   \right).
	\end{align}
	Let us define new moments $(\widetilde{\mu}_n)_{n\in\mathbb{Z}}$ as 
	given in
	\eqref{eq:def_new_moment}.
	We define 
	a matrix 
	$\widetilde{{\rm L}}$ as
	$\widetilde{{\rm L}}\triangleq(\widetilde{\mu}_{i-\ell j})_{i,j=0}^\infty$,
	which satisfies
	$\widetilde{{\rm L}}={\rm S}^{-1}{\rm L}={\rm diag}(S^{-1},S^{-1},\ldots,){\rm L}$. 
	Additionally, 
	the matrices $\widetilde{{\rm P}}$ and $\widetilde{{\rm L}}$
	satisfies 
		$
		\widetilde{{\rm P}}\widetilde{{\rm L}}
		={\rm PS}{\rm S}^{-1}{\rm L}
		={\rm PL}
		={\rm H}$.
	According to Theorem \ref{theorem:Favard}, 
	it follows that 
	$(\widetilde{\mu}_n)_{n\in\mathbb{Z}}$ is the moments associated with  $(\widetilde{P}_n(x))_{n=0}^\infty$.
\end{proof}

As an additional note,
we point out a
connection between $\ell$-LBPs and 
\emph{vector orthogonal polynomials of dimension $-d$} defined by Brezinski and Van Iseghem \cite{BV}.
It is shown in \cite{WCY} that $\ell$-LBPs admit a determinant expression
\begin{align}
	P_n(x)=\frac{1}{\tau_{n}}\left|
	\begin{array}{cccccc}
		\mu_0&\mu_{-\ell}&\mu_{-2\ell}&\cdots&\mu_{-(n-1)\ell}&1\\
		\mu_1&\mu_{-\ell+1}&\mu_{-2\ell+1}&\cdots&\mu_{-(n-1)\ell+1}&x\\
		\mu_2&\mu_{-\ell+2}&\mu_{-2\ell+2}&\cdots&\mu_{-(n-1)\ell+2}&x^2\\
		\vdots&\vdots&\vdots&&\vdots&\vdots\\
		\mu_n&\mu_{-\ell+n}&\mu_{-2\ell+n}&\cdots&\mu_{-(n-1)\ell+n}&x^n
	\end{array}
	\right|,
\end{align}
where $\mu_{n}\ (n\in\mathbb{Z})$ are the moments and 
$\tau_n\triangleq \det(\mu_{i-\ell j})_{i,j=0}^{n-1}$ is a 
block Toeplitz determinant of the moments.
The polynomials $Q_n(x)$ defined as
\begin{align}
	Q_n(x)\triangleq \frac{1}{\tau_{n}}\left|
	\begin{array}{ccccc}
		\mu_0&\mu_{-\ell}&\mu_{-2\ell}&\cdots&\mu_{-n\ell}\\
		\mu_1&\mu_{-\ell+1}&\mu_{-2\ell+1}&\cdots&\mu_{-n\ell+1}\\
		\mu_2&\mu_{-\ell+2}&\mu_{-2\ell+2}&\cdots&\mu_{-n\ell+2}\\
		\vdots&\vdots&\vdots&&\vdots\\
		\mu_{n-1}&\mu_{-\ell+n-1}&\mu_{-2\ell+n-1}&\cdots&\mu_{-n\ell+n-1}\\
		1&x&x^2&\cdots&x^n
	\end{array}
	\right|
\end{align}
is called the bi-orthogonal partner of $P_n(x)$ since they satisfy the bi-orthogonality
\begin{align}
	\mathcal{L}[P_n(x)Q_m(x)]=h_{n,m}\delta_{n,m},\qquad 
	n,m\in\mathbb{N}_0.
\end{align}
for non-vanishing $h_{n,m}$.
This is why $P_n(x)$ is called bi-orthogonal polynomials.
The polynomial $Q_n(x)$ is essentially the same as the polynomial 
introduced as a vector orthogonal polynomial of dimension $-d$ in \cite{BV}.

\section{Interpretation of $\ell$-LBPs using $\ell$-Schr\"oder paths}
\label{section:3}

This chapter is dedicated to interpreting the moments of $\ell$-LBPs
as generating functions of
$\ell$-Schr\"oder paths and another type of paths called
\emph{dual $\ell$-Schr\"oder paths}.
We focus specifically on primitive cases, 
as Theorem \ref{theorem:moment_primitive} tells that the moments of 
general cases can be expressed in terms of those of the primitive ones.

The moments $\mu_n=\mathcal{L}[x^n]$ of primitive $\ell$-LBPs, as defined by Theorem 
\ref{theorem:Favard}, are Laurent polynomials of 
the coefficients of the recurrence \eqref{eq:recurrence}.
For example, when $\ell=3$, the first few moments of non-negative degrees are
\begin{align}
	&\mu_0=1,\ \mu_1=\mu_2=0,\ \mu_3=a_{3,0},\ \mu_{4}=a_{3,0}a_{4,1},
	\ \mu_{5}=a_{3,0}a_{4,2}a_{4,1},\\
	&\mu_{6}=a_{3,0}(a_{1,0}a_{4,1}a_{4,2}+a_{4,1}a_{4,2}a_{4,3}+a_{2,0}a_{4,1}+a_{3,0}),\ldots
\end{align}
and those of negative degrees are
\begin{align}
	&\mu_{-1}=0,\ 
	\mu_{-2}=-\frac{a_{4,1}}{a_{3,1}},\ 
	\mu_{-3}=\frac{a_{2,0}a_{4,1}}{a_{3,1}a_{3,0}}+\frac{1}{a_{3,0}},\ 
	\mu_{-4}=\frac{a_{4,1}a_{4,2}}{a_{3,1}a_{3,2}}\ \\
	&\mu_{-5}=
	-\frac{a_{4,1}^{2} a_{2,0}}{a_{3,1}^{2} a_{3,0}}
	-\frac{a_{4,1} a_{2,1} a_{4,2}}{a_{3,1}^{2} a_{3,2}}
	-\frac{a_{4,1}}{a_{3,1}^{2}}
	-\frac{a_{4,1}}{a_{3,1} a_{3,0}},\ldots.
\end{align}
Henceforth, all the examples in this paper will be mentioned for the case where $\ell=3$.
In Theorem \ref{theorem:main1} in Section \ref{section:3.3},
this moments will be interpreted as generating functions.

Let us define some terminologies for lattice paths.
A \emph{step} is a vector of integers in the $x$-$y$ plane.
Let $S$ be a finite set of steps.
A \emph{path} 
is a word 
over $S$ 
equipped with an \emph{initial point}, a point in $\mathbb{Z}^2$.
We say a path is from its initial point.
Let $\eta=\eta_1\cdots\eta_r$, where $\eta_i\in S$, be a path from $(s,t)$.
The \emph{starting point} of the $i$-th step, denoted as ${\rm sp}_i$, is defined by
${\rm sp}_i\triangleq (s,t)+\sum_{j=1}^{i-1}\eta_j$ for $i=1,\ldots,r$.
The \emph{ending point} of the $i$-th step is defined by 
${\rm sp}_i+\eta_i$ for $i=1,\ldots,r$.
The \emph{final point} of a path is the ending point of its last step.
We say that a path goes to its final point.
A path whose word is empty is called the \emph{empty path} and is denoted by $\varepsilon$.

A weight of steps is a function $w\colon S\times\mathbb{Z}^2\to\mathbb{K}$.
Given a weight of steps, the weight of a path $\eta=\eta_1\cdots\eta_r$ is defined 
by the product of the weight of each step
as
$w(\eta)\triangleq\prod_{i=1}^{r} w(\eta_i,{\rm sp}_i)$.
The weight of the empty path is $1$.

For any paths $\eta=\eta_1\cdots\eta_r$ and $\nu=\nu_1\cdots\nu_c$, we define their concatenation as
$\eta\nu\triangleq \eta_1\cdots\eta_r\nu_1\cdots\nu_c$, with
the initial point being the same as that of $\eta$. 
We also define the concatenation of a path and a step as follows.
Let $\eta$ be a path and $x$ be a step.
Then define $\eta x\triangleq \eta_1\cdots\eta_rx$ and $x\eta\triangleq x\eta_1\cdots\eta_r$.
The initial points of $x\eta$ and $\eta x$ are the same as that of $\eta$.

For a set $A$ of paths and for $\alpha$ a path or a step, we define their concatenation as
$\alpha A\triangleq \bigcup_{\eta\in A}\alpha\eta$ and 
$A\alpha\triangleq \bigcup_{\eta\in A}\eta\alpha$.

\subsection{Combinatorial interpretation of the recurrence of $\ell$-LBPs}
A \emph{pre-Favard path} 
is a path from $(0,-\ell+1)$ with the following $2\ell+1$ types of steps: 
\begin{align}
	{\rm step}&\ \chi_i\triangleq (i,i),\qquad \text{for}\ i=1,\ldots,\ell,\\
	{\rm step}&\ \alpha_i\triangleq(0,i),\qquad \text{for}\ i=1,\ldots,\ell,\\
	{\rm step}&\ \alpha_{\ell+1}\triangleq(\ell,\ell+1).
\end{align}
We define the weight of the steps $\chi_i$ as $1$, and  
the weight of $\alpha_i$ as
\begin{align}
	w(\alpha_i,(s,t))\triangleq -a_{i,t+i-1},\qquad \forall(s,t)\in\mathbb{Z},\ i=1,\ldots,\ell+1,
\end{align}
depending on the height where they are placed.

A \emph{Favard path} is a pre-Favard path that satisfies the following two conditions.
\begin{itemize}
	\item
		If a step $\chi_i\ (i=1,\ldots,\ell-1)$ exists in a path, 
		it must be placed at the beginning of the path.
	\item
		The weight of a path does not depend on any of 
		$a_{i,j}\ (j<0)$.
\end{itemize}
Note that, from the first condition, a Favard path can contain at most one copy of $\chi_i\ (i=1,\ldots,\ell-1)$.
Let the \emph{width} $\w(\eta)$ and the \emph{height} of a path $\eta=\eta_1\cdots\eta_r$ 
be the $x$-component and the $y$-component of the vector $\eta_1+\cdots+\eta_r$.
Let $\mathcal{F}_n$ be the set of Favard path of height $n$.
\begin{figure}[htbp]
  \begin{minipage}[b]{0.48\linewidth}
    \centering
	\includegraphics{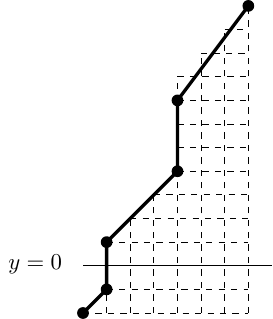}
  \end{minipage}
  \begin{minipage}[b]{0.48\linewidth}
    \centering
	\includegraphics{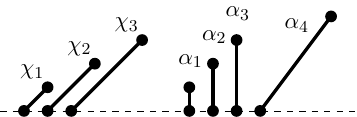}
  \end{minipage}
    \caption{
		Left: a Favard path $\eta$, whose
		steps are $\eta=\chi_1\alpha_2\chi_3\alpha_3\alpha_4$,
		weight $w(\eta)=-a_{2,0}a_{3,6}a_{4,10}$,
		width $\w(\eta)=7$ and height $13$.
	Right: the seven types of steps of Favard path.}
  \label{fig:Favard}
\end{figure}
Figure \ref{fig:Favard} shows an example of Favard path and the available steps for when $\ell=3$. 

The following theorem shows that primitive $\ell$-LBPs are generating functions of Favard paths.
\begin{theorem}
	\label{theorem:P_combin}
	Let $(P_n(x))_{n=0}^\infty$ be defined by the recurrence \eqref{eq:recurrence} with an initial condition
	$P_n(x)=x^n\ (n=0,\ldots,\ell-1)$.
	For any non-negative integer $n\in\mathbb{N}_0$, we have 
 	\begin{align}
		\label{eq:path_combinat}
		\sum_{\eta\in\mathcal{F}_n}w(\eta)x^{\w(\eta)} =P_n(x).
	\end{align}
\end{theorem}
\begin{proof}
	The cases $n\leq \ell-1$ are clear since $\mathcal{F}_n=\{\chi_n\}$ for $1\leq n\leq \ell-1$.
	When $n\geq \ell$,
	the set $\mathcal{F}_n$ is decomposed into a disjoint union of $\ell+2$ sets 
	as 
	\begin{align}
		\mathcal{F}_n=\left(\bigsqcup_{i=1}^\ell \mathcal{F}_{n-i}\alpha_{i}\right)
		\sqcup\mathcal{F}_{n-\ell}\chi_{\ell}
		\sqcup\mathcal{F}_{n-\ell-1}\alpha_{\ell+1}.
	\end{align}
	This shows that the polynomials $\sum_{\eta\in\mathcal{F}_n}w(\eta)x^{\w(\eta)}$ and $P_n(x)$ satisfies
	the same recurrence and the initial condition.
\end{proof}

\subsection{$\ell$-Schr\"oder paths and dual $\ell$-Schr\"oder paths}
We have defined $\ell$-Schr\"oder paths in Definition \ref{definition:1}.
Let us define the other path called dual $\ell$-Schr\"oder paths
having steps $\mathsf{A}_i\ (i=0,\ldots,\ell+1)$ as
\begin{align}
	\mathsf{A}_i&\triangleq\left\{
		\begin{array}{ll}
			(\ell-i,\ell-i),&{\rm for}\ i=0,\ldots,\ell-1,\\
			(\ell,0),&{\rm for}\ i=\ell,\\
			(\ell-1,-1),&{\rm for}\ i=\ell+1.
	\end{array}\right.
\end{align}
A  dual $\ell$-Schr\"oder path
is a path never going below the $x$-axis 
with the steps $\{\mathsf{A}_i\mid i=0,\ldots,\ell+1\}$.
These paths also essentially coincide with Schr\"oder paths when $\ell=1$.
\begin{figure}[htbp]
  \begin{minipage}[b]{0.74\linewidth}
    \centering
	\includegraphics{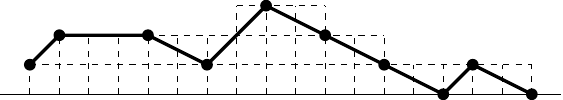}
  \end{minipage}
  \begin{minipage}[b]{0.25\linewidth}
    \centering
	\includegraphics{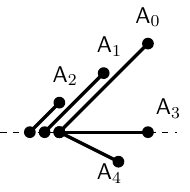}
  \end{minipage}
    \caption{
		Left: an example of dual $3$-Schr\"oder path.
	Right: the five types of steps of $3$-Schr\"oder path. }
  \label{fig:dualSchroder}
\end{figure}
Figure \ref{fig:dualSchroder} shows an example of dual $\ell$-Schr\"oder path 
and availavel steps when $\ell=3$.

Let us 
define a weight of the step $\mathsf{a}_0$ of $\ell$-Schr\"oder paths to be $1$ and
the weight of the step $\mathsf{a}_i\ (i=1,\ldots,\ell+1)$ as
\begin{align}
	w(\mathsf{a}_i,(s,t))\triangleq a_{i,t},\qquad\forall(s,t)\in\mathbb{Z}^2,\ i=1,\ldots,\ell+1.
\end{align}
Also let the weight of step $\mathsf{A}_i\ (i=0,\ldots,\ell+1)$ of dual $\ell$-Schr\"oder paths be 
\begin{align}
	w(\mathsf{A}_i,(s,t))\triangleq\left\{
		\begin{array}{ll}
			\vspace{3pt}
			\displaystyle\frac{(-1)^{\ell+1}}{a_{\ell,t}},&\textrm{if } i=0,\\
			\vspace{3pt}
			\displaystyle\frac{(-1)^{\ell+1+i}a_{i,t}}{a_{\ell,t}},&\textrm{if }
			1\leq i\leq \ell-1\ \textrm{or } i=\ell+1,\\
			\displaystyle\frac{1}{a_{\ell,t}},&\textrm{if }i=\ell.
		\end{array}
		\right.
\end{align}
For example, 
the weight of $\ell$-Schr\"oder path $\omega$ in Figure \ref{fig:Schroder} is
\begin{align}
w(\omega)=a_{1,1}^2a_{2,0}a_{3,1}a_{4,1}a_{4,2}^2a_{4,3}^3a_{4,4}a_{4,5},
\end{align}
and the weight of dual $\ell$-Schr\"oder path $\widetilde{\omega}$ in Figure \ref{fig:dualSchroder} is 
\begin{align}
	w(\widetilde{\omega})=-\frac{a_{2,1} a_{4,2}^{2} a_{1,1} a_{4,3} a_{4,1}^{2} a_{2,0}}
		{a_{3,1}^{4} a_{3,2}^{3} a_{3,3} a_{3,0}}.
	\end{align}
Note that when $\ell=3$,
the weights of the steps 
$\mathsf{A}_0,\mathsf{A}_1,\mathsf{A}_2,\mathsf{A}_3,\textrm{and }\mathsf{A}_4$
are
\begin{align}
\frac{1}{a_{3,t}},-\frac{a_{1,t}}{a_{3,t}},\frac{a_{2,t}}{a_{3,t}},\frac{1}{a_{3,t}},
\textrm{and }\frac{a_{4,t}}{a_{3,t}},
\end{align}
respectively, if their starting points are on the line $y=t$.

Let $r$ be the remainder of dividing $n$ by $\ell$.
Let $\mathcal{B}_n$ denote the set of $\ell$-Schr\"oder
paths from $(0,r)$ to $(n,0)$
and let
$\widetilde{\mathcal{B}}_n$ the set of dual $\ell$-Schr\"oder paths
from $(0,\ell-r)$ to $(n,0)$.
The example in Figure \ref{fig:Schroder} is in $\mathcal{B}_{16}$, and 
the example in Figure \ref{fig:dualSchroder} is in $\widetilde{\mathcal{B}}_{17}$.
We denote the weighted generating functions of the sets $\mathcal{B}_n$ and $\widetilde{\mathcal{B}}_n$ as
\begin{align}
	w(\mathcal{B}_n)\triangleq\sum_{\omega\in\mathcal{B}_n}w(\omega),\qquad{\rm and}\qquad
	w(\widetilde{\mathcal{B}}_n)\triangleq\sum_{\omega\in\widetilde{\mathcal{B}}_n}w(\omega).
\end{align}
\begin{remark}
A $\ell$-Schr\"oder path from $(s_1,s_2)$ to $(t_1,t_2)$ exists only when
$s_1+s_2\equiv t_1+t_2\pmod \ell$ and a 
dual $\ell$-Schr\"oder path only when
$s_1-s_2\equiv t_1-t_2\pmod\ell$.
\end{remark}
\begin{remark}
	Note that some sequences in the on-line encyclopedia of integer sequences \cite{OEIS} can be interpreted by
	$\mathcal{B}_n$ and $\widetilde{\mathcal{B}}_n$. For instance, when we set $\ell=2$,
	$(|\mathcal{B}_{n}|)_{n=0,1,2,\ldots}$ is coincident with the sequence A107708 and
	$(|\widetilde{\mathcal{B}}_n|)_{n=0,1,2,\ldots}$ with
	A007863.
\end{remark}

\subsection{Combinatorial interpretation of the moment of $\ell$-LBPs}
\label{section:3.3}
The following is the main theorem of the first half of this paper, providing
an explicit expression for the moments 
of primitive $\ell$-LBPs.
Let us denote
the remainder of dividing $p$ by $q$
as $(p\bmod q)$ for 
$p\in\mathbb{Z}$ and $q\in\mathbb{N}$.
\begin{theorem}
	\label{theorem:main1}
	Let $(P_n(x))_{n=0}^\infty$ be defined by the recurrence \eqref{eq:recurrence} with an initial condition
	$P_n(x)=x^n\ (n=0,\ldots,\ell-1)$.
	Let $\mathcal{L}\colon\mathbb{K}[x,x^{-1}]\to\mathbb{K}$ 
	be the unique
	linear functional such that $(P_n(x))_{n=0}^\infty$ is an
	$\ell$-LBPs with respect to $\mathcal{L}$, and that satisfies $\mathcal{L}[1]=1$.
	Then the moments $\mu_n\triangleq\mathcal{L}[x^n]$ are written explicitly as 
\begin{align}
	\label{eq:def_comb_moment}
	\mu_n=\left\{
		\begin{array}{ll}
			a_{\ell,0}w(\mathcal{B}_{n-\ell}),&\textrm{if }\ell\leq n\\
			0,&\textrm{if }1\leq n\leq \ell-1,\\
			(-1)^{n\bmod \ell}w(\widetilde{B}_{-n}),&\textrm{if }n\leq 0.
		\end{array}
		\right.
\end{align}
\end{theorem}
The remaining of this section is devoted to the proof of Theorem \ref{theorem:main1}.
We show the orthogonality $\mathcal{L}[P_n(x)x^{-k\ell}]=h_n\delta_{n,k}$
for the functional $\mathcal{L}$ defined by the moments \eqref{eq:def_comb_moment}.
Then, the claim of Theorem \ref{theorem:main1} follows 
from the uniqueness of the functional.
We show the orthogonality by 
decomposing it into
two lemmas: one is related to the moments of positive degrees (Lemma \ref{lemma:positive}), 
and the other is related to non-positive degrees (Lemma \ref{lemma:negative}).
After two lemmas, we present the orthogonality in a general form (Theorem \ref{theorem:g_orth}).
\begin{remark}
	Theorem \ref{theorem:main1} generalizes Kamioka's combinatorial 
	interpretation for the moments of $1$-LBPs \cite[Theorem 3.1]{K2}.
\end{remark}
\begin{lemma}
	\label{lemma:positive}
	Let $\mathcal{L}$ be the linear functional defined by the moments \eqref{eq:def_comb_moment}.
	For any non-negative integer $n\in\mathbb{N}_0$
	and integer $k\in\mathbb{Z}$, we have
\begin{subequations}  \label{eq:positive_lemma}
    \begin{empheq}[left = {
		\mathcal{L}\left[[x^{>0}]P_n(x)x^{-\ell k}\right]=
	\empheqlbrace \,}]{alignat = 2}
        & 
		a_{\ell,0}
		\sum_{\substack{\eta\in\mathcal{F}_n,\\
		\eta \text{\rm\ starts with }\chi_\ell,\\ 
		\w(\eta)=\ell(k+1)}}
		w(\eta),
		&\qquad&   \text{if $k\geq 0$,}  \label{eq:positive_case1}  \\
        & 
	a_{\ell,0}
	\sum_{\substack{\omega\colon\ell\textrm{\rm -Schr\"oder}\\
\text{\rm from }(0,n)\\
\text{\rm to }(n-\ell(k+1),0)}}w(\omega),
		&     &   \text{if $k<0$},  \label{eq:positive_case2}
    \end{empheq}
\end{subequations}
	where $[x^{>0}]P_n(x)x^{-\ell k}$ denotes the sum of terms of degree one or heiger in the 
	Laurent polynomial $P_n(x)x^{-\ell k}$.
	The sum \eqref{eq:positive_case1} on the right-hand side is taken over all Favard paths in $\mathcal{F}_n$
	whose first steps are $\chi_\ell$
	and whose widths are $\ell(k+1)$.
	The sum \eqref{eq:positive_case2} is taken over all $\ell$-Schr\"oder paths 
	from $(0,n)$ to $(n-\ell(k+1),0)$.
\end{lemma}
\begin{proof}
	By Theorem \ref{theorem:P_combin} and the definition \eqref{eq:def_comb_moment} of the moments,
	the left-hand side of \eqref{eq:positive_lemma} is written as
	\begin{align}
\mathcal{L}\left[[x^{>0}]P_n(x)x^{-\ell k}\right]
=
a_{\ell,0}\sum_{(\eta,\omega)\in I(n,k)}w(\eta)w(\omega),
	\end{align}
	where
	\begin{align}
		I(n,k)\triangleq\{(\eta,\omega)\mid \eta\in\mathcal{F}_n,\ \w(\eta)\geq\ell(k+1),\ 
		\omega\in\mathcal{B}_{\w(\eta)-\ell(k+1)}\}.
	\end{align}

	Let us define the \emph{prefixes} to interpret the right-hand side of \eqref{eq:positive_lemma}.
	Let $q$ and $r$ be the quotient and the remainder 
	in the $h$ division by $\ell$, where $h$ is a non-negative integer.
	The \emph{Favard prefix} of length $h$, denoted by $f(h)$, is the path 
	from $(0,-\ell+1)$ defined as 
	$f(h)\triangleq \chi_\ell^q$ if $r=0$ and 
	$f(h)\triangleq\chi_r\chi_\ell^q$ if $r>0$,
	where $\chi_\ell^q$ means $\chi_\ell$ repeated $q$ times.
	The \emph{Schr\"oder prefix} of length $h$, denoted by $s(h)$ is the path 
	from $(0,r)$
	defined as $s(h)\triangleq \mathsf{a}_0^q$.

	Let 
	$J_1(n,k),J_2(n,k)$ and $J(n,k)$ be subsets of $I(n,k)$ as 
	\begin{align}
		J_1(n,k)&\triangleq\{(\eta,\varepsilon)
			\mid \eta\in\mathcal{F}_n,\ 
		\text{\rm\ the first step of }\eta\text{\rm\ is }\chi_\ell,\ \w(\eta)=\ell(k+1)\}
		\textrm{ and }\\
			J_2(n,k)&\triangleq
			\{(f(n),s(n)\omega')\mid 
			\text{$\omega'$ is a $\ell$-Schr\"oder path,\ }
			s(n)\omega'\in\mathcal{B}_{n-\ell(k+1)}\},
	\end{align}
	and $J(n,k)\triangleq J_1(n,k)\cup J_2(n,k)$.
	The right-hand side of \eqref{eq:positive_lemma} is the sum over $J(n,k)$ multiplied by $a_{\ell,0}$.
	Indeed when $n\neq0$, we have $J_2(n,k)=\emptyset$ if $k\geq 0$, and $J_1(n,k)=\emptyset$ if $k< 0$,
	hence the sum over $J(n,k)$ is reduced to \eqref{eq:positive_case1} if $k\geq 0$, and 
	\eqref{eq:positive_case2} if $k< 0$.

	Our task is to construct a sign-reversing 
	involution $\varphi$ on the set $I(n,k)\setminus J(n,k)$.
	The involution we construct is similar to 
	Viennot's \cite{V} but a bit complicated.
	For readers unfamiliar with this type of involutions,
	we recommend \cite{CKS} for a detailed description.

	For any Favard path $\eta\in\mathcal{F}_n\ (n\in\mathbb{N}_0)$, 
	let $h_\eta\in\mathbb{N}_0$ be the maximum integer such that $\eta=f(h_\eta)\eta_1$
	for some path $\eta_1$.
	Also, for any $\ell$-Schr\"oder path $\omega\in\mathcal{B}_n\ (n\in\mathbb{N}_0)$, 
	let $h_\omega\in\mathbb{N}_0$ be the maximum integer
	such that $\omega=s(h_{\omega})\omega_1$.

	The involution $\varphi$ works as follows.
	Let $(\eta,\omega)\in I(n,k)\setminus J(n,k)$.
	It compares the values $h_\eta$ and $h_\omega$.
	\begin{description}
		\item[Case 1: When $h_{\eta}>h_{\omega}$.]\mbox{}\\
			We can write $\eta$ and $\omega$ as
			$\eta=f(h_{\omega})\mu_1$ and $\omega=s(h_\omega)\omega_1$
			for some path $\eta_1$ and $\omega_1$.
			Then $\omega_1\neq\varepsilon$ because $(\eta,\omega)\notin J_1(n,k)$.
			The first step of $\omega_1$ is $\mathsf{a}_i$ for some $i\in\{1,\ldots,\ell+1\}$ 
			because of the maximality of $h_\omega$.
			Also, the first step of $\eta_1$ is $\chi_\ell$ because of 
			the maximality of $h_\eta$ and $h_\eta>h_\omega$.
			Thus we can write $(\eta,\omega)$ as $(\eta,\omega)
			=(f(h_\omega)\chi_\ell\eta_2,s(h_\omega)\mathsf{a}_i\omega_2)$
			for some paths $\eta_2$ and $\omega_2$.
			Then $h_\omega+\ell-i\geq 0$.
			Define the value of $\varphi$ as $\varphi(\eta,\omega)\triangleq
			(f(h_\omega+\ell-i)\alpha_i\eta_2,s(h_\omega+\ell-i)\omega_2)$.
		\item[Case 2: When $h_\eta\leq h_\omega$.]\mbox{}\\
			Write $\eta$ and $\omega$ as
			$\eta=f(h_{\eta})\eta_1$ and $\omega=s(h_\eta)\omega_1$.
			Then $\eta_1\neq\varepsilon$ because $(\eta,\omega)\notin J_2(n,k)$.
			The first step of $\eta_1$ is $\alpha_i$ for some $i\in\{1,\ldots,\ell+1\}$
			hence we can write $\eta=f(h_\eta)\alpha_i\eta_2$.
			We have $h_\eta+i-\ell\geq 0$ since $\eta$ is a Favard path.
			Define the value of $\varphi$ as $\varphi(\eta,\omega)\triangleq
			(f(h_\eta+i)\eta_2,s(h_\eta+i-\ell)\mathsf{a}_i\omega_1)$.
	\end{description}
	One can check that $\varphi(\eta,\omega)\in I(n,k)\setminus J(n,k)$ and that
	$\varphi$ is an involution.
	To check $\varphi$ is sign-reversing, let $(\eta',\omega')\triangleq\varphi(\eta,\omega)$.
	Suppose we are in Case 1.
	Let $\eta_2$ and $\omega_2$ be the paths that satisfies 
	$(\eta,\omega)=(t(h_\omega)\chi_\ell\eta_2,s(h_\omega)\mathsf{a}_i\omega_2)$ and 
	suppose that the initial point of $\eta_2$ and $\omega_2$ are
 	the final point of $f(h_\omega)\chi_\ell$ and 
	$s(h_\omega)\mathsf{a}_i$, respectively.
	We have
	\begin{align}
		w(\eta)w(\omega)=w(\eta_2)a_{i,h_\omega}w(\omega_2).
	\end{align}
	On the other hand, we have
	\begin{align}
		w(\eta')w(\omega')
		=-a_{i,(-\ell+1+h_\omega+\ell-i)+i-1}w(\eta_2)w(\omega_2)=-a_{i,h_\omega}w(\eta_2)w(\omega_2).
	\end{align}
	Thus we get $w(\eta')w(\omega')=-w(\eta)w(\omega)$.
	A similar proof works for Case 2.
\end{proof}
\begin{figure}[htbp]
  \begin{minipage}[b]{0.4\linewidth}
    \centering
	\includegraphics{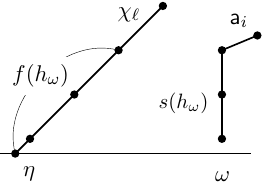}
  \end{minipage}
  \begin{minipage}[b]{0.19\linewidth}
    \centering
	$\underset{1:1}{\overset{\varphi}{\longleftrightarrow}}$
  \end{minipage}
  \begin{minipage}[b]{0.4\linewidth}
    \centering
	\includegraphics{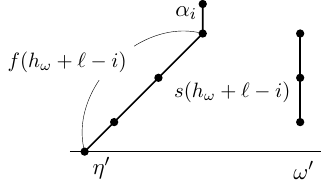}
  \end{minipage}
  \caption{how the involution $\varphi$ works in the proof of Lemma \ref{lemma:positive}.}
  \label{fig:inv1}
\end{figure}
Figure \ref{fig:inv1} illustrates how the involution $\varphi$ in the proof of Lemma \ref{lemma:positive}
changes the paths.
The figure represents the beginning part of the paths $\eta$ and $\omega$ where
$(\eta,\omega)\in I(n,k)\setminus J(n,k)$.
When we are In Case 1, the
path's beginning parts are like the figure's left side.
The involution $\varphi$ changes them to the right side
while leaving the remaining parts unchanged.

\begin{lemma}
	\label{lemma:negative}
	Let $\mathcal{L}$ be the linear functional defined by the moments \eqref{eq:def_comb_moment}.
	For any non-negative integer $n\in\mathbb{N}_0$ and integer $k\in\mathbb{Z}$, we have
\begin{subequations}  \label{eq:negative_lemma}
    \begin{empheq}[left = {
		\mathcal{L}\left[[x^{\leq 0}]P_n(x)x^{-\ell k}\right]=
	\empheqlbrace \,}]{alignat = 2}
        & 
	\sum_{\substack{\eta\in\mathcal{F}_{n},\\
	\eta{\rm\ starts\ with\ }\alpha_{\ell},\\
	\w(\eta)=\ell k}}
	w(\eta),
		&\qquad&   \text{if $k<n$,}  \label{eq:negative_case1}  \\
        & 
	(-1)^n
	\sum_{\substack{\omega\colon\text{\rm dual }\ell\text{\rm -Schr\"oder}\\
	\text{\rm from }(0,n)\\\text{\rm to }(\ell k-n,0)}}
	w(\omega),
		&     &   \text{if $k\geq n$},  \label{eq:negative_case2}
    \end{empheq}
\end{subequations}
	where $[x^{\leq 0}]P_n(x)x^{-\ell k}$ denotes the sum of terms of degree equal to or less than zero
	in the 
	Laurent polynomial $P_n(x)x^{-\ell k}$.
	The sum \eqref{eq:negative_case1} on the right-hand side is taken over all Favard paths in $\mathcal{F}_n$
	whose first steps are $\alpha_\ell$
	and whose widths are $\ell k$.
	The sum \eqref{eq:negative_case2} is taken over all dual $\ell$-Schr\"oder paths
	from $(0,n)$ to $(\ell k-n,0)$.
\end{lemma}
\begin{proof}
	The left-hand side of the equation \eqref{eq:negative_lemma}
	can be written as 
	\begin{align}
		\mathcal{L}\left[[x^{\leq 0}]P_n(x)x^{-\ell k}\right]
		=
		\sum_{(\eta,\omega)\in\widetilde{I}(n,k)}w(\eta)w(\omega)(-1)^{\w(\eta)\bmod\ell},
	\end{align}
	where
	\begin{align}
		\widetilde{I}(n,k)\triangleq\{(\eta,\omega)\mid
		\eta\in\mathcal{F}_n,\ \w(\eta)\leq \ell k,\ 
	\omega\in\widetilde{B}_{\ell k-\w(\eta)}\}.
	\end{align}
	Let $q$ and $r$ be the quotient and the remainder of dividing $h$ by $\ell$.
	We define the \emph{dual-Favard prefix} $F(h)$ as the path from $(0,-\ell+1)$ with 
 	$F(h)\triangleq \alpha_\ell^q$ if $r=0$ and 
	$F(h)\triangleq \chi_r\alpha_\ell^q$ if $r>0$.
	Also, define the dual-Schr\"oder prefix $S(h)$ as the path from $(0,r)$ with steps
	$S(h)\triangleq \mathsf{A}_0^q$.

	Define subsets $\widetilde{J}_1(n,k),\widetilde{J}_2(n,k)$
	and $\widetilde{J}(n,k)$ of $\widetilde{I}(n,k)$ as
	\begin{align}
		\widetilde{J}_1(n,k)&\triangleq
		\{(\eta,\varepsilon)\mid \eta\in\mathcal{F}_n,\ 
		{\rm the\ first\ step\ of\ }\eta{\rm\ is\ } \alpha_\ell, \w(\eta)=\ell k\},\\
		\widetilde{J}_2(n,k)&\triangleq
		\{(F(n),S(n)\omega')\mid \text{$\omega'$ is a dual $\ell$-Schr\"oder path,\ }
		S(n)\omega'\in\widetilde{\mathcal{B}}_{\ell k-r}
		\},
	\end{align}
	and $\widetilde{J}(n,k)\triangleq \widetilde{J}_1(n,k)\cup J_2(n,k)$.
	The right-hand side of \eqref{eq:negative_lemma} is the weighted sum over $\widetilde{J}(n,k)$
	since when $n\neq0$, we have $\widetilde{J}_1(n,k)=\emptyset$ if $k\geq n$ and 
	$\widetilde{J}_2(n,k)=\emptyset$ if $k < n$.

	For any $(\eta,\omega)\in \widetilde{I}(n,k)\setminus\widetilde{J}(n,k)$, let $h_{\eta}$
	(resp. $h_{\omega}$) be the maximum integer such that
	$\eta=F(h_{\eta})\eta_1$ for some paht $\eta_1$
	(resp. $\omega=S(h_{\omega})\omega_1$ for some path $\omega_1$).
	Let us define a sign-reversing involution $\widetilde{\varphi}$ on  the set 
	$\widetilde{I}(n,k)\setminus\widetilde{J}(n,k)$, similar to that in Lemma \ref{lemma:positive}.
	It compares the values
	$h_\eta$ and $h_\omega$.
	\begin{description}
		\item[Case 1: When $h_{\eta}>h_{\omega}$.]\mbox{}\\
			We can write $\eta$ and $\omega$ as
			$\eta=F(h_{\omega})\eta_1$ and $\omega=S(h_\omega)\omega_1$
			using some path $\eta_1$ and $\omega_1$.
			Then $\omega_1\neq\varepsilon$ since $(\eta,\omega)\notin \widetilde{J}_1(n,k)$.
			The first step of $\omega_1$ is $\mathsf{A}_i$ for some $i\in\{1,\ldots,\ell+1\}$ 
			because of the maximality of $h_\omega$.
			Also, the first step of $\eta_1$ is $\alpha_\ell$ because of 
			the maximality of $h_\eta$ and $h_\eta>h_\omega$.
			Thus we can write $(\eta,\omega)$ as $(\eta,\omega)
			=(F(h_\omega)\alpha_\ell\eta_2,s(h_\omega)\mathsf{A}_i\omega_2)$
			for some paths $\eta_2$ and $\omega_2$.
			Define the value of $\widetilde{\varphi}$ as
			\begin{align}
				\widetilde{\varphi}(\eta,\omega)\triangleq
				\left\{
				\begin{array}{ll}
					(F(h_\omega)\chi_\ell\eta_2,S(h_\omega)\omega_2),
					&{\rm if}\ i=\ell,\\
					(F(h_\eta+\ell-i)\alpha_{i}\eta_2,S(h_\eta+\ell-i)\omega_2),\quad&
					{\rm if}\ i\neq \ell.
				\end{array}
					\right.
		\end{align}
		\item[Case 2: When $h_\eta\leq h_\omega$.]\mbox{}\\
			Write as $\eta=F(h_{\eta})\eta_1$ and $\omega=S(h_\eta)\omega_1$.
			Then $\eta_1\neq\varepsilon$ since $(\eta,\omega)\notin \widetilde{J}_2(n,k)$, hence 
			let the first step of $\eta_1$ be $\sigma$. From the maximality of $h_\eta$,
			we have $\sigma\neq\alpha_\ell$, that is, 
 			$\sigma\in\{\alpha_1,\alpha_2,\ldots,\alpha_{\ell-1},\alpha_{\ell+1},\chi_\ell\}$.
			We can write $(\eta,\omega)$ as $(\eta,\omega)=(F(h_\omega)\sigma\eta_2,S(h_\eta)\omega_1)$.
			Define the value of $\widetilde{\varphi}$ as
			\begin{align}
				\widetilde{\varphi}(\eta,\omega)\triangleq
				\left\{
				\begin{array}{ll}
					(F(h_\eta)\alpha_\ell\eta_2,S(h_\eta)\mathsf{A}_\ell\omega_1),
					&{\rm if}\ \sigma=\chi_\ell,\\
					(F(h_\eta+i-\ell)\alpha_{\ell}\eta_2,S(h_\eta+i-\ell)\mathsf{A}_i\omega_1),\quad&
					{\rm if}\ \sigma=\alpha_i\ {\rm for\ some\ }i\neq \ell.
				\end{array}
					\right.
		\end{align}
	\end{description}
		Let $(\eta^{\ast},\omega^{\ast})\triangleq \widetilde{\varphi}(\eta,\omega)$.
		We show $\widetilde{\varphi}(\eta,\omega)\in \widetilde{I}(n,k)\setminus\widetilde{J}(n,k)$
		by showing 
		\begin{align}
			\label{eq:wwww}
		\w(\eta)+\w(\omega)=\w(\eta^{\ast})+\w(\omega^{\ast}).
		\end{align}
		If \eqref{eq:wwww} holds,
		then from $\omega\in\widetilde{\mathcal{B}}_{\ell k-\w(\eta)}$, we have
		$\w(\omega^\ast)=\ell k-\w(\eta^\ast)$
		and from $\w(\omega^\ast)\geq 0$, we have $\ell k\geq \w(\eta^\ast)$;
		Thus we have $(\eta^\ast,\omega^\ast)\in\widetilde{I}(n,k)$ if \eqref{eq:wwww} holds.
		The proof of \eqref{eq:wwww} is the following. When in Case 1 and when
		$(\eta,\omega)=(F(h_\omega)\alpha_\ell\eta_2,S(h_\omega)\mathsf{A}_\ell\omega_2)$,
		we have
		\begin{align}
			\left\{
				\begin{array}{l}
					\w(\eta)=r+\w(\eta_2),\\
					\w(\omega)=\ell q+\ell+\w(\omega_2),
	\end{array}
\right.\ 
{\rm and}\ 
	\left\{
		\begin{array}{l}
			\w(\eta^\ast)=r+\ell+\w(\eta_2),\\
			\w(\omega^\ast)=\ell q+\w(\omega_2),
	\end{array}
	\right.
	\end{align}
	where the quotient and the remainder of dividing $h_\omega$ by $\ell$ are
	denoted as $q$ and $r$, respectively, 
	hence \eqref{eq:wwww} holds.
	A similar proof works for the remaining cases.
	It is easy to check that $\widetilde{\varphi}$ is an involution.
	The remaining task is to check $\widetilde{\varphi}$ is sign-reversing.
	When in Case 1 and when 
	$(\eta,\omega)=(F(h_\omega)\alpha_\ell\eta_2,S(h_\omega)\mathsf{A}_\ell\omega_2)$,
	it suffices to show that 
	\begin{align}
		\label{eq:wwww2}
		w(F(h_\omega)\alpha_\ell)w(S(h_\omega)\mathsf{A}_\ell)=
		-w(F(h_\omega)\chi_\ell)w(S(h_\omega)).
	\end{align}
	By substituting
	\begin{align}
		\label{eq:wFh}
		w(F(h_\omega)\alpha_\ell)&=\left(\prod_{j=0}^{q-1}(-a_{\ell,r+j\ell})\right)(-a_{\ell,h_\omega}),\\
		w(S(h_\omega)\mathsf{A}_\ell)&=\left(\prod_{j=0}^{q-1}\frac{(-1)^{\ell+1}}{a_{\ell,r+j\ell}}\right)
		\frac{1}{a_{\ell,h_\omega}},\\
		w(F(h_\omega)\chi_\ell)&=\left(\prod_{j=0}^{q-1}(-a_{\ell,r+j\ell})\right),\quad\textrm{and}\\
		w(S(h_\omega))&=\left(\prod_{j=0}^{q-1}\frac{(-1)^{\ell+1}}{a_{\ell,r+j\ell}}\right),
	\end{align}
	one can check that \eqref{eq:wwww2} holds.
	When 
	$(\eta,\omega)=(F(h_\omega)\alpha_\ell\eta_2,S(h_\omega)\mathsf{A}_i\omega_2)$ for $i\neq \ell$,
	let us write $q^\ast$ and $r^\ast$ the quotient and the remainder of dividing $h_\omega+\ell-i$ by $\ell$,
	respectively.
	It suffices to show that 
	\begin{align}
		w(F(h_\omega)\alpha_\ell)w(S(h_\omega)\mathsf{A}_i)(-1)^r=
		-w(F(h_\omega+\ell-i)\alpha_\ell)w(S(h_\omega+\ell-i))(-1)^{r^\ast}.\\
		\label{eq:wwww3}
	\end{align}
	By substituting
	\begin{align}
		w(S(h_\omega)\mathsf{A}_i)&=\left(\prod_{j=0}^{q-1}\frac{(-1)^{\ell+1}}{a_{\ell,r+j\ell}}\right)
		\frac{(-1)^{\ell+1+i}a_{i,h_\omega}}{a_\ell,h_\omega},\\
		w(F(h_\omega+\ell-i)\alpha_i)&=\left(\prod_{j=0}^{q^\ast-1}(-a_{\ell,r^\ast+j\ell})\right)(-a_{i,h_\omega}),\\
		w(S(h_\omega+\ell-i))&=\prod_{j=0}^{q^\ast-1}\frac{(-1)^{\ell+1}}{a_{\ell,r^\ast-j\ell}},
	\end{align}
	and \eqref{eq:wFh},
	one can check that \eqref{eq:wwww3} holds.
	A similar proof works for Case2. 
 \end{proof}
\begin{figure}[htbp]
  \begin{minipage}[b]{0.43\linewidth}
    \centering
	\includegraphics{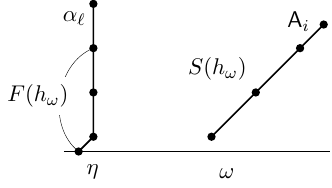}
  \end{minipage}
  \begin{minipage}[b]{0.03\linewidth}
    \centering
	$\underset{1:1}{\overset{\widetilde{\varphi}}{\longleftrightarrow}}$
  \end{minipage}
  \begin{minipage}[b]{0.45\linewidth}
    \centering
	\includegraphics{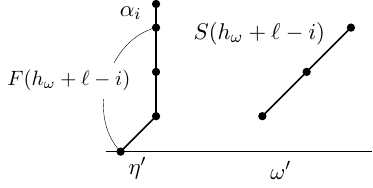}
  \end{minipage}
  \caption{how the involution $\widetilde{\varphi}$ works
  in the proof of Lemma \ref{lemma:negative}.}
  \label{fig:inv2}
\end{figure}
Figure \ref{fig:inv2} illustrates how the involution $\widetilde{\varphi}$
in the proof of Lemma \ref{lemma:negative}
changes the paths.
In Case 1, 
the beginning parts of the paths are like the left side. 
The involution $\widetilde{\varphi}$ changes them as the right side
while leaving the remaining parts unchanged.

Combining Lemma \ref{lemma:positive} and Lemma \ref{lemma:negative},
we obtain the general orthogonality.
\noeqref{eq:g_orth1,eq:g_orth2,eq:g_orth3}
\begin{theorem}
	\label{theorem:g_orth}
	Let $\mathcal{L}$ be the linear functional defined by the moments \eqref{eq:def_comb_moment}.
	For any non-negative integer $n\in\mathbb{N}_0$ and integer $k\in\mathbb{Z}$, we have
\begin{subequations}  \label{eq:g_orth}
    \begin{empheq}[left = {
		\mathcal{L}\left[P_n(x)x^{-\ell k}\right]=
	\empheqlbrace \,}]{alignat = 2}
        & 
	(-1)^n
	\sum_{\substack{\omega\colon\text{\rm dual }\ell\text{\rm -Schr\"oder}\\
	\text{\rm from }(0,n)\text{\rm to }(\ell k-n,0)}}
	w(\omega),
		& \qquad    &   \text{if $n\leq k$},  \label{eq:g_orth1}\\
		&0,&&\text{if $0\leq k\leq n-1$},\label{eq:g_orth2}\\
        & 
	a_{\ell,0}
	\sum_{\substack{\omega\colon\ell\textrm{\rm -Schr\"oder}\\
\text{\rm from }(0,n)
\text{\rm to }(n-\ell(k+1),0)}}w(\omega),
		&     &   \text{if $k<0$}.  \label{eq:g_orth3}
    \end{empheq}
\end{subequations}
More specifically, we have
\begin{align}
	\mathcal{L}[P_n(x)x^{-n\ell}]=(-1)^n\prod_{i=1}^n\frac{a_{\ell+1,i}}{a_{\ell,i}}
\end{align}
when $k=n$, and 
\begin{align}
	\mathcal{L}[P_n(x)x^{\ell}]=a_{\ell,0}\prod_{i=1}^na_{\ell+1,i}
\end{align}
when $k=-1$.
\end{theorem}
\begin{proof}
	Substitute \eqref{eq:positive_lemma} and \eqref{eq:negative_lemma} into
	 \begin{align}
		 \mathcal{L}[P_n(x)x^{-\ell k}]
		 =
		 \mathcal{L}[[x^{>0}]P_n(x)x^{-\ell k}]
		 +\mathcal{L}[[x^{\leq0}]P_n(x)x^{-\ell k}].
	 \end{align}
	We get \eqref{eq:g_orth1} since $\mathcal{L}[[x^{>0}]P_n(x)x^{-\ell k}]=0$ if $n\leq k$, 
	get \eqref{eq:g_orth3} since $\mathcal{L}[[x^{\leq 0}]P_n(x)x^{-\ell k}]=0$ if $k<0$, and 
	\eqref{eq:g_orth2} since $\mathcal{L}[[x^{>0}]P_n(x)x^{-\ell k}]
		 +\mathcal{L}[[x^{\leq0}]P_n(x)x^{-\ell k}]=0$ if $0\leq k\leq n-1$.
	When $k=n$, the right-hand side of \eqref{eq:g_orth} is the weight of 
	the pair $(\eta,\omega)=(F(n),S(n)\mathsf{A}_{\ell+1}^n)$.
	When $k=-1$, the right-hand side is the wheigh of 
	$(\eta,\omega)=(f(n),s(n)\mathsf{a}_{\ell+1}^n)$.
\end{proof}
\section{Non-intersecting tuplets of $\ell$-Schr\"oder paths}
\label{section:4}
In the latter half of this paper,
we study enumerative aspects of $\ell$-Schr\"oder paths and 
new paths called 
\emph{small $\ell$-Schr\"oder paths}.
The results are applied to obtain 
closed-form generating functions for non-intersecting tuplets of $\ell$-Schr\"oder paths,
the second aim of this paper.

Before we get to the main point,
we would like to briefly discuss \emph{Narayana polynomials} 
to relate the results of this chapter to the $\ell$-LBPs.
Recall that the conventional Narayana polynomials \cite{K2,BSS} are the generating functions
of Schr\"oder paths concerning the number of steps $(2,0)$.
We define the $n$-th $\ell$-Narayana polynomial $N_n(a_1,\ldots,a_{\ell+1})$ as 
\begin{align}
	N_n=N_n(a_{1},\ldots,a_{\ell+1})\triangleq\sum_{\omega\in\mathcal{B}_{n}}
	\prod_{i=1}^{\ell+1}a_i^{(\text{number of steps $\mathsf{a_i}$ in $\omega$})}.
\end{align}
The 
moments of positive degrees of
primitive $\ell$-LBPs are written by $\ell$-Narayana polynomials as
$\mu_{n+\ell}=a_\ell N_n\ (n\geq 0)$
when the recurrence coefficients in \eqref{eq:recurrence} are constants
as $a_{i,j}\triangleq a_i$ for all $i=1,\ldots,\ell+1$ and $j\in\mathbb{N}_0$.
Using the results of this chapter, we will 
show the following two property of $\ell$-Narayana polynomials.
\begin{theorem}
	\label{theorem:Narayana_root}
	For any integer $n\geq \ell$, the $n$-th $\ell$-Narayana polynomial is
	divisible by $\sum_{i=0}^{\ell}a_ia_{\ell+1}^{\ell-i}$,
	where $a_0\triangleq 1$.
\end{theorem}
\begin{theorem}
	\label{corollary:det_moment}
	The following block Hankel determinant of $\ell$-Narayana polynomials can be factorized as
	\begin{align}
		\det(N_{\ell(i+1)+j})_{i,j=0}^{n-1}=
		a_{\ell+1}^{\binom{n}{2}}
		\left(\sum_{i=0}^\ell a_ia_{\ell+1}^{\ell-i}\right)^{\binom{n+1}{2}},
	\end{align}
	where $a_0\triangleq 1$.
\end{theorem}

\begin{remark}
	Theorem \ref{theorem:Narayana_root} reduces to Proposition 2.1 in \cite{BSS} when $\ell=1$,
	which describes roots of Narayana polynomials.
\end{remark}

\subsection{Relation between big and small $\ell$-Schr\"oder paths}
We introduce two new paths and consider their relations with $\ell$-Schr\"oder paths.
Firstly, let us define a new weight $v$ on $\ell$-Schr\"oder paths, more suitable for the 
aim of this chapter.
The weight of the step $\mathsf{a}_i\ (0\leq i\leq \ell)$ from the point $(x,y)\in\mathbb{Z}^2$ is defined as 
$v(\mathsf{a}_i,(x,y))\triangleq b_{i,j}$,
and the weight of $\mathsf{a}_{\ell+1}$ from $(x,y)$ is defined as 
$v(\mathsf{a}_{\ell+1},(x,y))\triangleq c_j$,
where $j\triangleq (x+y)/\ell$.

Let us denote by $\mathcal{B}_{s,t}^r$ the set of $\ell$-Schr\"oder paths from $(s\ell-r,r)$ to $(t\ell,0)$.
The generating function for the set $\mathcal{B}_{s,t}^r$ is denoted by
$v(\mathcal{B}_{s,t}^r)\triangleq\sum_{\omega\in\mathcal{B}_{s,t}^r}v(\omega)$.

A \emph{small $\ell$-Schr\"oder path} is an $\ell$-Schr\"oder path that does not have 
any step $\mathsf{a}_i$ starting from 
in the region
$y\leq i-1$ for all 
$i=1,\ldots,\ell$.
Let us denote by $\mathcal{S}_{s,t}^r$ the set of small $\ell$-Schr\"oder paths from $(s\ell-r,r)$ to $(t\ell,0)$.
We refer to  $\ell$-Sch\"order paths as \emph{big} $\ell$-Schr\"oder paths. 
A \emph{marked $\ell$-Schr\"oder path} is a big $\ell$-Schr\"oder path with a mark on its eastern slope,
that is, an element of the set
\begin{align}
	\mathcal{M}_{s,t}^r\triangleq
	\{(\omega,m)\in\mathcal{B}_{s,t}^r\times\mathbb{N}_0\mid
	\omega \textrm{ have a step starting or ending at }(t\ell-m,m)\},
\end{align}
if $s\neq t$ or $r>0$, and $\mathcal{M}_{t,t}^{0}\triangleq\{(\varepsilon,0)\}$,
if $s=t$ and $r=0$.
See Fugire \ref{fig:MS} for examples of marked (left-side)
and small (right-side) $\ell$-Schr\"oder paths.

A weight $v$ of marked $\ell$-Schr\"oder paths is defined as, for $(\omega,m)\in\mathcal{M}_{s,t}^r$,
\begin{align}
	v((\omega,m))\triangleq v(\omega)\frac{c_{t+1}^m}{c_t^m}.
\end{align}
The generating functions of $\mathcal{S}_{s,t}^r$ and $\mathcal{M}_{s,t}^r$ are denoted by
$v(\mathcal{S}_{s,t}^r)\triangleq \sum_{\omega\in\mathcal{S}_{s,t}^r}v(\omega)$ and 
$v(\mathcal{M}_{s,t}^r)\triangleq \sum_{(\omega,m)\in\mathcal{M}_{s,t}^r}v((\omega,m))$, respectively.
The following lemmas describe the relation between generating functions 
$v(\mathcal{B}_{s,t}^r)$, $v(\mathcal{S}_{s,t}^r)$ and $v(\mathcal{M}_{s,t}^r)$.
\begin{lemma}
	\label{lemma:BM}
	For any integers $s,t$ with $s\leq t$ and  non-negative integer $r$ with $0\leq r<\ell$, we have
	\begin{align}
		v(\mathcal{B}_{s,t+1}^r)=v(\mathcal{M}_{s,t}^r)\sum_{i=0}^\ell b_{i,t}c_{t+1}^{\ell-i}.
	\end{align}
\end{lemma}
\begin{proof}
	The proof is based on a weight-preserving bijection
	$\theta\colon\mathcal{B}_{s,t+1}^r\to\mathcal{M}_{s,t}^r\times\mathcal{B}_{t,t+1}^0$.
	The map $\theta$ is defined as follows.
	Let $\omega\in\mathcal{B}_{s,t+1}^r$.
	Tracing the path $\omega$ backward from the final step, there is a step $\mathsf{a}_{\ell-k}$
	other than $\mathsf{a}_{\ell+1}$ that is encountered first.
	Hence we can write uniquely as 
	$\omega=\omega_1\mathsf{a}_{\ell-k}\mathsf{a}_{\ell+1}^{k+m}$ for some path $\omega_1$
	and non-negative integers $m,k$.
	Define $\theta$ as 
	$\theta(\omega)\triangleq((\omega_1\mathsf{a}_{\ell}^m,m),\mathsf{a}_{\ell-k}\mathsf{a}_{\ell+1}^{k})
	\in\mathcal{M}_{s,t}^r\times\mathcal{B}_{t,t+1}^0$.
\end{proof}
\begin{figure}[htbp]
  \begin{minipage}[b]{0.4\linewidth}
    \centering
	\includegraphics{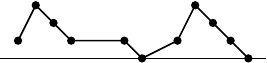}
  \end{minipage}
  \begin{minipage}[b]{0.15\linewidth}
    \centering
	$\overset{\theta}{\underset{1:1}{\longrightarrow}}$
\end{minipage}
  \begin{minipage}[b]{0.4\linewidth}
    \centering
	\includegraphics{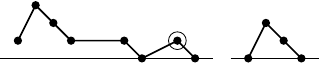}
  \end{minipage}
  \caption{
	  Left: a big $3$-Schr\"oder path $\omega\in\mathcal{B}_{0,5}^{1}$.
  Right: the pair $\theta(\omega)=((\omega',m),\omega'')\in\mathcal{M}_{0,4}^1\times\mathcal{B}_{4,5}^0$}
  \label{fig:BM}
\end{figure}
Figure \ref{fig:BM} shows how the bijection $\theta$ in Lemma \ref{lemma:BM} works.
The left side is a big $3$-Schr\"oder path $\omega\in\mathcal{B}_{0,5}^1$ with steps
$\omega=\mathsf{a}_1\mathsf{a}_4^2\mathsf{a}_3\mathsf{a}_4\mathsf{a}_2\mathsf{a}_1\mathsf{a}_4^3$.
The right side is a pair $((\omega',m),\omega'')$ of a
marked $3$-Schr\"oder path $(\omega',m)\in\mathcal{M}_{0,4}^1$ and 
a big $3$-Schroder path $\omega''\in\mathcal{B}_{4,5}^0$,
where the steps are
$\omega'=\mathsf{a}_1\mathsf{a}_4^2\mathsf{a}_3\mathsf{a}_4\mathsf{a}_2\mathsf{a}_4$, 
$\omega''=\mathsf{a}_1\mathsf{a}_4^2$. and $m=1$.
The bijection $\theta$ maps $\omega$ onto $((\omega',m),\omega'')$.

\begin{lemma}
	\label{lemma:MS}
	For any integers $s,t$ with $s\leq t$ and  non-negative integer $r$ with $0\leq r<\ell$, we have
	\begin{align}
		\label{eq:MS}
		v(\mathcal{M}_{s,t}^r)=
		c_{t+1}^{-\ell}\sum_{i=0}^rc_s^{r-i}
		\left(v(\mathcal{S}_{s,t}^{i+\ell})\mid_{c_{k}:=c_{k+1}}\right).
	\end{align}
	The symbol $v(\mathcal{S}_{s,t}^{i+\ell})|_{c_{k}:=c_{k+1}}$ in \eqref{eq:MS} represents
	the generating function $v(\mathcal{S}_{s,t}^{i+\ell})$ in which all the variables $c_{k}$
	replaced by the variable $c_{k+1}$ 
	simultaneously
	for all $k\in\mathbb{Z}$.
\end{lemma}
\begin{proof}
	The statement is obvious if $s=t$.
	Let us consider the case when $s < t$.
	We decompose the set $\mathcal{M}_{s,t}^r$ into $r+1$ disjoint subsets based on
	the number of steps $\mathsf{a}_{\ell+1}$
	placed at the beginning of paths.
	Let $\mathcal{N}_{s,t}^r$ be the subset of $\mathcal{M}_{s,t}^r$ defined as 
	\begin{align}
		\mathcal{N}_{s,t}^r\triangleq
		\{(\omega,m)\mid (\omega,m)\in\mathcal{M}_{s,t}^r,\ {\rm the\ first\ step\ of\ } \omega
		{\rm\ is\ not\ }\mathsf{a}_{\ell+1}\}.
	\end{align}
	Using $\mathcal{N}_{s,t}^r$, the set $\mathcal{M}_{s,t}^r$ can be decomposed as
	\begin{align}
		\label{eq:MN}
		\mathcal{M}_{s,t}^r=\bigsqcup_{i=0}^r\mathsf{a}_{\ell+1}^{r-i}\mathcal{N}_{s,t}^i,
	\end{align}
	where $\mathsf{a}_{\ell+1}^{r-i}\mathcal{N}_{s,t}^i$ denotes the set of 
	marked $\ell$-Schr\"oder paths in $\mathcal{M}_{s,t}^r$ whose first $r-i$ steps are $\mathsf{a}_{\ell+1}$ and
	$(r-i+1)$-th step is not $\mathsf{a}_{\ell+1}$.
	From \eqref{eq:MN}, writing the generating function of $\mathcal{N}_{s,t}^r$ as 
	$v(\mathcal{N}_{s,t}^r)\triangleq\sum_{(\omega,m)\in\mathcal{N}_{s,t}^r}v((\omega,m))$,
	we have
	\begin{align}
		\label{eq:MNwt}
		v(\mathcal{M}_{s,t}^r)=\sum_{i=0}^rc_s^{r-i}v(\mathcal{N}_{s,t}^i).
	\end{align}
	Our next step is to construct 
	a bijection that sends $(\omega,m)\in\mathcal{N}_{s,t}^i$
	to
	$\omega'\in\mathcal{S}_{s,t}^{i+\ell}$ for $0\leq i\leq r$ that satisfies
	\begin{align}
		\label{eq:NSwt}
		v((\omega,m))
		=
		\left.\left(\frac{1}{c_{t}^\ell}v(\omega')\right)\middle|_{c_{k}:=c_{k+1}}\right..
	\end{align}
	If such a bijection exists, then by summing \eqref{eq:NSwt} over $\mathcal{N}_{s,t}^r$ we have
	\begin{align}
		v(\mathcal{N}_{s,t}^i)=
		c_{t+1}^{-\ell}\left(v(\mathcal{S}_{s,t}^{i+\ell})\mid_{c_k:=c_{k+1}}\right)
	\end{align}
	and, together with \eqref{eq:MNwt}, we get \eqref{eq:MS}.
	The bijection is constructed as follows.

	We split the marked path $(\omega,m)\in\mathcal{N}_{s,t}^i$ into two paths,
	before and after the mark.
	The former is a path from $(s\ell-i,i)$ to $(t\ell-m,m)$, and the latter 
	is a path from $(t\ell-m,m)$ to $(t\ell,0)$.
	Let us denote the former as $\omega_1$ and the latter as $\omega_2$.
	The path $\omega_1$ has a unique decomposition as a sequence of the elements of
	$\{\mathsf{a}_j\mathsf{a}_{\ell+1}^k\mid 0\leq j\leq \ell, k\geq 0\}$.
	Let us write this decomposition as $\omega_1=\alpha_1\cdots\alpha_d$.
	We define $\omega'$ as the path from $(s\ell-i-\ell,i+\ell)$ with steps
	\begin{align}
		\omega'\triangleq\rt(\alpha_1)\cdots\rt(\alpha_d)\omega_2\mathsf{a}_{\ell+1}^\ell,
	\end{align}
	where $\rt(v_1v_2\cdots v_n)\triangleq v_2\cdots v_nv_1$.
	This map $\mathcal{N}_{s,t}^i\ni(\omega,m)\mapsto\omega'\in\mathcal{S}_{s,t}^{i+\ell}$ is a bijection
	and it satisfies \eqref{eq:NSwt}.
\end{proof}
\begin{figure}[htbp]
  \begin{minipage}[b]{0.4\linewidth}
    \centering
	\includegraphics{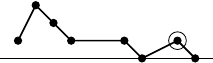}
  \end{minipage}
  \begin{minipage}[b]{0.15\linewidth}
    \centering
	$\underset{1:1}{\longrightarrow}$
\end{minipage}
  \begin{minipage}[b]{0.4\linewidth}
    \centering
	\includegraphics{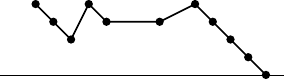}
  \end{minipage}
  \caption{Left: a marked $3$-Schr\"oder path in $\mathcal{N}_{0,4}^{1}$.
  Right: a small $3$-Schr\"oder path in $\mathcal{S}_{0,4}^4$.}
  \label{fig:MS}
\end{figure}
Figure \ref{fig:MS} shows an example of how the bijection in Lemma \ref{lemma:MS} works.
The left side is a marked $3$-Schr\"oder path $(\omega,1)\in\mathcal{N}_{0,4}^1$ with steps
$\omega=\mathsf{a}_1\mathsf{a}_4^2\mathsf{a}_3\mathsf{a}_4\mathsf{a}_2\mathsf{a}_4$.
The decomposition of $\omega$ is as $\omega=(\mathsf{a}_1\mathsf{a}_4^2)(\mathsf{a}_3\mathsf{a}_4)
(\mathsf{a}_2)\mathsf{a}_4$.
Therefore, the bijection maps $(\omega,1)$ onto the small $3$-Schr\"oder paths $\omega'\in\mathcal{S}_{0,4}^4$
on the right side of the figure.
The steps of $\omega'$ are
$\omega'=\mathsf{a}_4^2\mathsf{a}_1\mathsf{a}_4\mathsf{a}_3\mathsf{a}_2\mathsf{a}_4^4$.

Combining Lemma \ref{lemma:BM} and Lemma \ref{lemma:MS}, we obtain the following.
\begin{theorem}
	\label{theorem:BS}
	For any integers $s,t$ with $s\leq t$ and non-negative integer $r$ with $0\leq r<\ell$, we have
	\begin{align}
		\label{eq:BS}
		v(\mathcal{B}_{s,{t+1}}^r)=
		c_{t+1}^{-\ell}
		\left(\sum_{i=0}^rc_s^{r-i}\left(v(\mathcal{S}_{s,t}^{i+\ell})\middle|_{c_k:=c_{k+1}}\right)\right)
		\left(\sum_{i=0}^\ell b_{i,t}c_{t+1}^{\ell-i}\right).
	\end{align}
\end{theorem}
We obtain Theorem \ref{theorem:Narayana_root} mentioned at the beginning of the chapter
by setting the variables 
in Theorem \ref{theorem:BS} 
as $b_{i,j}:=a_i$ for $i=0,\ldots,\ell+1$ and $c_j:=a_{\ell+1}$.
Also by $b_{i,j}:=1$ and $c_j:=1$ for all $i,j\in\mathbb{Z}$ in
Theorem \ref{theorem:BS}, we obtain the following.
\begin{corollary}
	\label{corollary:BS}
	For any integers $s,t$ with $s\leq t$ and non-negative integer $r$ with $0\leq r<\ell$, we have
	\begin{align}
		|\mathcal{B}_{s,t+1}^r|=(\ell+1)\sum_{i=0}^r|\mathcal{S}_{s,t}^{i+\ell}|.
	\end{align}
\end{corollary}
Corollary \ref{corollary:BS} is reduced to the well-known fact that 
the $n$-th big Schr\"oder number is double the
$n$-th small Schr\"oder number when $\ell=1$ (cf. \cite{YJ}).
\subsection{Generating functions for non-intersecting tuplets of $\ell$-Schr\"oder paths}
In this subsection, we show that some generating functions of non-intersecting 
tuplets of
$\ell$-Schr\"oder paths are factorized nicely, as in the case of $\ell=1$.
We generalize the technique developed by Eu and Fu \cite{EuFu} for Schr\"oder paths.
Theorem \ref{theorem:BS} plays an essential role in the proof of the factorization.

Let us define some terms that we need for non-intersecting paths.
Let $\mathbf{s}=(s_0,\ldots,s_{n-1})$ and $\mathbf{t}=(t_0,\ldots,t_{n-1})$ be tuplets of 
points in $\mathbb{Z}^2$.
A \emph{path system} of size $n$ from $\mathbf{s}$ to $\mathbf{t}$ is a tuplet of paths
$(\omega_0,\ldots,\omega_{n-1})$
such that $\omega_i$ is from $s_i$ to $t_{\sigma(i)}$
for $0\leq i\leq n-1$, where $\sigma$ is a permutation on $\{0,1,\ldots,n-1\}$.
A path system is said to be \emph{non-intersecting} 
if the paths are pairwise disjoint, i.e., 
if no steps in the paths share their starting or ending point.
The weight $v(\pmb{\omega})$ of a path system $\pmb{\omega}=(\omega_0,\ldots,\omega_{n-1})$
is defined as 
$v(\pmb\omega)\triangleq\sign(\sigma)v(\omega_0)\cdots v(\omega_{n-1})$.

Let us denote as $\Pi_n$ the set of non-intersecting path systems of $n$ $\ell$-Schr\"oder paths
from $\mathbf{s}$ to $\mathbf{t}$ with
\begin{align}
	\mathbf{s}&\triangleq((-i,i\bmod \ell))_{i=0,\ldots,n-1}\qquad\textrm{and}\\
	\mathbf{t}&\triangleq(((i+1)\ell,0))_{i=0,\ldots,n-1},
\end{align}
where $i\bmod\ell$ denotes the remainder of dividing $i$ by $\ell$.
Also, denote as $\Omega_n$ the set of non-intersecting systems of $n$
small $\ell$-Schr\"oder paths from $\mathbf{s}$ to $\mathbf{t}$ with
\begin{align}
	\mathbf{s}&\triangleq((-i,(i\bmod \ell)+\ell))_{i=0,\ldots,n-1}\qquad\textrm{and}\\
	\mathbf{t}&\triangleq((i\ell,0))_{i=0,\ldots,n-1}.
\end{align}
See Figure \ref{fig:PiOmega} for examples of non-intersecting paths in 
$\Pi_n$ (left side) and $\Omega_n$ (right side).
Note that although a path system is non-intersecting, lines may cross when drawn in a plane.

Let us write the generating function of $\Pi_n$ and $\Omega_n$ as
\begin{align}
	v(\Pi_n)\triangleq\sum_{\pmb\omega\in\Pi_n}v(\pmb\omega),\qquad
	v(\Omega_n)\triangleq\sum_{\pmb\omega\in\Omega_n}v(\pmb\omega).
\end{align}
These generating functions have the following relations.
\begin{lemma}
	\label{lemma:PiOmega}
	For any non-negative integer $n$, we have
	\begin{align}
		\label{eq:PiOmega}
		v(\Pi_n)=\left(v(\Omega_n)\middle|_{c_k:=c_{k+1}}\right)\times
		\prod_{t=0}^{n-1}\left(c_{t+1}^{-\ell}\sum_{i=0}^{\ell}b_{i,t}c_{t+1}^{\ell-i}\right).
	\end{align}
\end{lemma}
\begin{proof}
	From the Lindstr\"om--Gessel--Viennot lemma \cite{GV}, we know that $v(\Pi_n)$ and $v(\Omega_n)$ are
	written as determinants of generating functions of big and small $\ell$-Schr\"oder paths as
	\begin{align}
		\label{eq:detB}
		v(\Pi_n)&=\det(v(\mathcal{B}_{-\lfloor i/\ell \rfloor,j+1}^{i\bmod\ell}))_{i,j=0}^{n-1},\\
		\label{eq:detS}
		v(\Omega_n)&=\det(v(\mathcal{S}_{-\lfloor i/\ell \rfloor,j}^{(i\bmod\ell)+\ell}))_{i,j=0}^{n-1},
	\end{align}
	where $\lfloor i/\ell \rfloor$ and $i\bmod\ell$ denotes the quotient and the remainder of 
	$i$ divided by $\ell$.
	Substituting \eqref{eq:BS} into \eqref{eq:detB}, we have
	\begin{align}
		\label{eq:detB_henkei}
		\begin{aligned}
		&\det(v(\mathcal{B}_{-\lfloor i/\ell\rfloor,j+1}^{i\bmod\ell}))_{i,j=0}^{n-1}\\
		&=
		\left(\prod_{t=0}^{n-1}c_{t+1}^{-\ell}\sum_{i=0}^\ell b_{i,t}c_{t+1}^{\ell-i}\right)
		\det\left(\sum_{p=0}^{i\bmod\ell}c_{-\lfloor i/\ell\rfloor}^{(i\bmod\ell)-p}
		\left(v(\mathcal{S}_{-\lfloor i/\ell\rfloor,j}^{p+\ell})\middle|_{c_k:=c_{k+1}}\right)
	\right)_{i,j=0}^{n-1}.
		\end{aligned}
	\end{align}
	By performing the elementary row operations on the determinant in \eqref{eq:detB_henkei},
	we obtain
	\begin{align}
		\label{eq:detB_henkei2}
		\det(v(\mathcal{B}_{-\lfloor i/\ell\rfloor,j+1}^{i\bmod\ell}))_{i,j=0}^{n-1}
		=
		\left(\prod_{t=0}^{n-1}c_{t+1}^{-\ell}\sum_{i=0}^\ell b_{i,t}c_{t+1}^{\ell-i}\right)
		\det\left(v(\mathcal{S}_{\lfloor i/\ell\rfloor,j}^{(i\bmod\ell)+\ell})
		|_{c_k:=c_{k+1}}\right)_{i,j=0}^{n-1}.
	\end{align}
	From \eqref{eq:detB_henkei2} and \eqref{eq:detS}, we obtain \eqref{eq:PiOmega}.
\end{proof}
The generating functions also satisfy the following.
\begin{lemma}
	\label{lemma:OmegaPi}
	For any positive integer $n$, we have
	\begin{align}
		\label{eq:OmegaPi}
		v(\Omega_n)=c_0^\ell\left(\prod_{i=1}^{n-1}c_{i}^{\ell+1}\right)
	\left(\prod_{i=1}^{\lfloor(n-1)/\ell\rfloor}b_{0,-i}\right)
	v(\Pi_{n-1}).
\end{align}
\end{lemma}
\begin{proof}
	There is a natural bijection between $\Pi_{n-1}$ and $\Omega_{n}$ that
	shift all paths in $\pmb{\omega}\in\Pi_{n-1}$ northwestward.
	Let $\pmb{\omega}=(\omega_0,\ldots,\omega_{n-2})\in\Pi_{n-1}$.
	The bijection is defined as $\pmb{\omega}\mapsto 
	\pmb{\omega}'\triangleq(\mathsf{a}_{\ell+1}^{\ell},\omega_0',\ldots,\omega_{n-2}')\in\Omega_n$ where
	the word for $\omega'_i$ is defined as
	\begin{align}
\omega_i'\triangleq
		\left\{
			\begin{array}{ll}
			\omega_i\mathsf{a}_{\ell+1}^{\ell+1}&{\rm if\ }i\bmod \ell\not\equiv \ell-1,\\
			\mathsf{a}_{0}\omega_i\mathsf{a}_{\ell+1}^{\ell+1}&{\rm if\ }i\bmod \ell\equiv \ell-1.
			\end{array}
		\right.
	\end{align}
	The weight $v(\pmb{\omega}')$ of the new path system satisfies
	\begin{align}
		v(\pmb{\omega}')=
		c_0^\ell\left(\prod_{i=1}^{n-1}c_{i}^{\ell+1}\right)
	\left(\prod_{i=1}^{\lfloor(n-1)/\ell\rfloor}b_{0,-i}\right)
	v(\pmb{\omega}).
	\end{align}
	Summing up over $\Pi_{n-1}$ leads to \eqref{eq:OmegaPi}.
Figure \ref{fig:PiOmega} shows an example of the bijection.
\end{proof}
\begin{figure}[htbp]
  \begin{minipage}[b]{0.4\linewidth}
    \centering
	\includegraphics{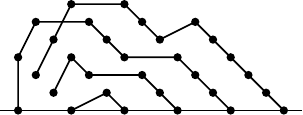}
  \end{minipage}
  \begin{minipage}[b]{0.06\linewidth}
    \centering
	$\underset{1:1}{\longrightarrow}$
\end{minipage}
  \begin{minipage}[b]{0.4\linewidth}
    \centering
	\includegraphics{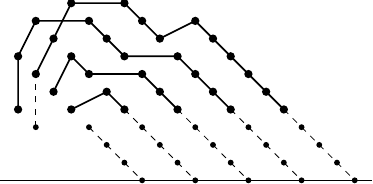}
  \end{minipage}
  \caption{Left: a 
	  non-intersecting tuplet of big $3$-Schr\"oder paths in the set $\Pi_4$.
  Right: a non-intersecting tuplet of small $3$-Schr\"oder paths in the set $\Omega_5$.
The dotted lines are frozen steps due to the non-intersectingness.}
  \label{fig:PiOmega}
\end{figure}

Combining Lemma \ref{lemma:PiOmega} and Lemma \ref{lemma:OmegaPi}, we obtain
a recurrence for $v(\Pi_n)$.
\begin{lemma}For any positive integer $n$, we have
	\begin{align}
		\label{eq:Pi_recurrence}
		v(\Pi_n)=
		\left(\prod_{i=2}^n c_i\right)
		\left(\prod_{i=1}^{\lfloor(n-1)/\ell\rfloor}b_{0,-i}\right)
		\left(\prod_{t=0}^{n-1}\sum_{i=0}^\ell b_{i,t}c_{t+1}^{\ell-i}\right)
		(v(\Pi_{n-1})|_{c_{k}:=c_{k+1}}).
	\end{align}
\end{lemma}
From the recurrence \eqref{eq:Pi_recurrence}, we have an explicit formula for $v(\Pi_n)$.
This is the main theorem of the latter half of this paper.
\begin{theorem}
	\label{theorem:Pi_explicit}
	For any non-negative integer $n$, we have
	\begin{align}
		v(\Pi_n)=
		\left(\prod_{i=2}^nc_i^{i-1}\prod_{j=1}^{\lfloor(i-1)/\ell\rfloor}b_{0,-j}\right)
		\left(\prod_{s=1}^{n-1}\prod_{t=0}^{s-1}\sum_{i=0}^\ell b_{i,t}c_{s}^{\ell-i}\right).
	\end{align}
\end{theorem}
As a direct consequence of Theorem \ref{theorem:Pi_explicit}, 
we can show Theorem 
\ref{corollary:det_moment}
presented at the beginning of the chapter.
\begin{proof}[Proof of Theorem \ref{corollary:det_moment}]
	Reduce the variables appearing in the weight $v$ 
	as $b_{i,j}:=a_i$ and $c_j:=a_{\ell+1}$ for all $j\in\mathbb{Z}$.
	Then we have
 	$v(\mathcal{B}^{i\bmod\ell}_{-\lfloor i/\ell\rfloor,j+1})=
	w(N_{i+\ell(j+1)})$.
	Thus we have 
	$\det(N_{\ell(i+1)+j})_{i,j=0}^{n-1}
	=\det(v(\mathcal{B}^{i\bmod\ell}_{-\lfloor i/\ell\rfloor,j+1}))_{i,j=0}^{n-1}
	=v(\Pi_n)$.
\end{proof}
\begin{remark}
We obtain Theorem 6.14 in \cite{KimStanton} by setting $\ell=2$ in Theorem \ref{theorem:Pi_explicit}. 
It is worth highlighting that distinct orthogonal functions, 
namely $R_I$ polynomials and $2$-LBPs, 
result in the same factorization formula.
\end{remark}
\begin{remark}
	By considering the one-to-one correspondence between the domino tilings of the 
	Aztec diamonds and the non-intersecting paths in $\Pi_n$ when $\ell=1$,
	one will find that Theorem \ref{theorem:Pi_explicit} provides 
	another proof of 
	refined Aztec diamond theorem that first appeared in Ciucu~\cite[Theorem 3.1]{C}.
\end{remark}
\begin{remark}
	The problem of finding tiling models with a one-to-one correspondence between 
	the non-intersecting paths in
	$\Pi_n$ for $\ell\geq 2$ is still open.
	It would be interesting to seek such tiling models.
\end{remark}
\section*{Acknowledgments}
The author would like to thank Professor Shuhei Kamioka for 
helpful discussions and comments.

\end{document}